\documentclass[a4paper,11pt]{article}
\usepackage{braket,hyperref}
\usepackage{amssymb,mathtools,amsthm,amsmath}
\usepackage{tikz}
\usepackage{xargs}
\usepackage{enumitem}
\usepackage{verbatim}

\theoremstyle{plain}
\newtheorem{theorem}{\bf Theorem}

\newtheorem{conjecture}[theorem]{\bf Conjecture}
\newtheorem{proposition}[theorem]{\bf Proposition}
\newtheorem{corollary}[theorem]{\bf Corollary}
\newtheorem{lemma}[theorem]{\bf Lemma}
\theoremstyle{definition}

\newenvironment{remark}[1][Remark.]{\begin{trivlist}
		\item[\hskip \labelsep {\bfseries #1}]}{\end{trivlist}}

\numberwithin{theorem}{section}
\numberwithin{equation}{section}
\newcommand{\Rea}{{\mathbb R}}

\newcommand{\FF}{{\mathbb F}}

\DeclareMathOperator{\lk}{lk}

\newcommand{\lfrac}[2]{\left\lfloor\frac{#1}{#2}\right\rfloor}
\newcommand{\ufrac}[2]{\left\lceil\frac{#1}{#2}\right\rceil}

\newcommand{\flagcomplex}[2]{\text{Fl}_{#1,#2}}

\newcommand{\lap}[3]{\Delta_{#1}^{+}(\flagcomplex{#2}{#3})}

\newcommand{\gl}[1]{\text{GL}(#1)}

\newcommand{\bigoh}[1]{O\left(#1\right)}

\begin{document}

\title{Asymptotic behavior of Laplacian eigenvalues of subspace inclusion graphs}
\author{Alan Lew\footnote{ Dept. Math. Sciences, Carnegie Mellon University, Pittsburgh, PA 15213, USA. e-mail: alanlew@andrew.cmu.edu}}

	\date{}
	\maketitle

\begin{abstract}
  Let $\text{Fl}_{n,q}$ be the simplicial complex whose vertices are the non-trivial subspaces of $\mathbb{F}_q^n$ and whose simplices correspond to families of subspaces forming a flag. Let $\Delta^{+}_k(\text{Fl}_{n,q})$ be the $k$-dimensional weighted upper Laplacian on $ \text{Fl}_{n,q}$. The spectrum of $\Delta^{+}_k(\text{Fl}_{n,q})$ was first studied by Garland, who obtained a lower bound on its non-zero eigenvalues.
  
  Here, we focus on the $k=0$ case.  We determine the asymptotic behavior of the eigenvalues of $\Delta_{0}^{+}(\text{Fl}_{n,q})$ as $q$ tends to infinity. In particular, we show that for large enough $q$, $\Delta_{0}^{+}(\text{Fl}_{n,q})$ has exactly $\left\lfloor n^2/4\right\rfloor+2$ distinct eigenvalues, and that every eigenvalue $\lambda\neq 0,n-1$ of $\Delta_{0}^{+}(\text{Fl}_{n,q})$ tends to $n-2$ as $q$ goes to infinity. This solves the $0$-dimensional case of a conjecture of Papikian.
\end{abstract}

\section{Introduction}

Let $n\geq 3$ be an integer and $q$ a prime power. Let $\FF_q^n$ be the $n$-dimensional vector space over the finite field $\FF_q$. We say that a linear subspace $V\subset \FF_q^n$ is \emph{trivial} if $V=\{0\}$ or $V=\FF_q^n$. A family of nested non-trivial linear subspaces $V_1\subsetneq V_2\subsetneq \cdots \subsetneq V_k$ is called a \emph{flag}. 
A flag of length $n-1$ is called a \emph{complete flag}.

Let $\flagcomplex{n}{q}$ be the simplicial complex whose vertices correspond to non-trivial linear subspaces of $\mathbb{F}_q^n$, and whose simplices are the families of subspaces forming a flag. Note that the complete flags are exactly the maximal faces of $\flagcomplex{n}{q}$. In particular, for any prime power $q$,  $\flagcomplex{n}{q}$ is a pure $(n-2)$-dimensional complex.
 
Let $C^k(\flagcomplex{n}{q})$ be the space of real $k$-cochains on $\flagcomplex{n}{q}$, and let $d_k: C^k(\flagcomplex{n}{q})\to C^{k+1}(\flagcomplex{n}{q})$ be the $k$-th \emph{coboundary operator}.
Let $\flagcomplex{n}{q}(k)$ be the set of $k$-dimensional simplices of $\flagcomplex{n}{q}$. For $\sigma=\{V_1,\ldots,V_{k+1}\}\in \flagcomplex{n}{q}(k)$, let $w(\sigma)$ be the number of complete flags extending $\sigma$.
That is, $w(\sigma)$ is the number of maximal faces of $\flagcomplex{n}{q}$ containing $\sigma$.
We define an inner product on the vector space $C^k(\flagcomplex{n}{q})$ by
\[
    \langle \phi,\psi \rangle =\sum_{\sigma\in \flagcomplex{n}{q}(k)} w(\sigma)\phi(\sigma)\psi(\sigma).
\]
Let $d_k^*$ be the operator adjoint to $d_k$ with respect to this inner product.
We define the \emph{weighted upper $k$-Laplacian} $\lap{k}{n}{q}: C^k(\flagcomplex{n}{q})\to C^k(\flagcomplex{n}{q})$ by
\[
   \lap{k}{n}{q}= d_k^* d_k. 
\]

Identifying $\lap{k}{n}{q}$ with its matrix representation with respect to the ``standard basis" for $C^k(\flagcomplex{n}{q})$ (see Section \ref{sec:prelims_lap}), we  obtain the following explicit formula:

Let $V_1\subsetneq V_2\subsetneq\cdots\subsetneq V_k$, $W_1\subsetneq W_2\subsetneq\cdots\subsetneq W_k$ be flags in $\FF_q^n$, and let $\sigma=\{V_1,\ldots,V_k\}\in \flagcomplex{n}{q}$ and $\tau=\{W_1,\ldots,W_k\}\in\flagcomplex{n}{q}$. Assume that $|\sigma\cap \tau|=k-1$ and $\sigma\cup\tau\in\flagcomplex{n}{q}$. Let $U$ be the unique element in $\sigma\setminus \tau$ and $V$ be the unique element in $\tau\setminus \sigma$. We define $\epsilon(\sigma,\tau)$ as the number of elements of $\sigma\cap \tau$ lying between $U$ and $V$.

Let $V_0=\{0\}$ and $V_{k+1}=\FF_q^n$. Let $\eta=\sigma\setminus \{V_i\}$ for some $1\leq i\leq k$. We define $r(\sigma,\eta)=\dim(V_{i+1})-\dim(V_{i-1})$ and $t(\sigma,\eta)=\dim(V_{i})-\dim(V_{i-1})$. 

For $a\geq b\geq 0$, let $\binom{a}{b}_q$ be the number of $b$-dimensional subspaces contained in $\FF_q^a$ (see Section \ref{sec:bin_prelims} for more details and an explicit formula).

\begin{lemma}\label{lemma:laplacian_formula}
    Let $n\geq 3$ and let $q$ be a prime power. Let $k\geq 0$. Then $\lap{k}{n}{q}$ is a $|\flagcomplex{n}{q}(k)|\times|\flagcomplex{n}{q}(k)|$ matrix with entries
    \[
    \lap{k}{n}{q}_{\sigma,\tau}=\begin{cases}
                    n-k-2 & \text{if } \sigma=\tau,\\
                    (-1)^{\epsilon(\sigma,\tau)+1}\binom{r(\sigma\cup\tau,\sigma)}{t(\sigma\cup\tau,\sigma)}_q^{-1} & \text{if } |\sigma\cap \tau|=k,\, \sigma\cup\tau\in\flagcomplex{n}{q},\\
                    0 & \text{otherwise,}
    \end{cases}
    \]
    for all $\sigma,\tau\in \flagcomplex{n}{q}(k)$.
\end{lemma}
In the special case $k=0$, we obtain:
\begin{corollary}\label{cor:laplacian_formula_0_dim}
    Let $n\geq 3$ and let $q$ be a prime power. Then, $\lap{0}{n}{q}$ is a $|\flagcomplex{n}{q}(0)|\times|\flagcomplex{n}{q}(0)|$ matrix with entries
    \[
    \lap{0}{n}{q}_{U,V}=\begin{cases}
                    n-2 & \text{if } U=V,\\
                    - \binom{n-\dim(U)}{\dim(V)-\dim(U)}_q^{-1} & \text{if } U\subsetneq V,\\
                    - \binom{\dim(U)}{\dim(V)}_q^{-1} & \text{if } V\subsetneq U,\\
                    0 & \text{otherwise,}
    \end{cases}
    \]
    for every two non-trivial subspaces $U,V\subset \FF_q^n$.
\end{corollary}

Laplacian operators on simplicial complexes were introduced by Eckmann in \cite{eckmann1944harmonische}, where he proved a simplicial version of the Hodge Theorem, relating the kernel of the $k$-dimensional Laplacian of a simplicial complex to its $k$-th cohomology group.
Motivated by a conjecture of Serre  about the cohomology of certain groups associated with an algebraic group over a non-archimedean local field, Garland studied in \cite{garland1973p} the operator $\lap{k}{n}{q}$ and its eigenvalues. In particular, he proved the following result:
\begin{theorem}[Garland \cite{garland1973p}]\label{thm:garland}
Let $n\geq 3$ and $0\leq k\leq n-3$. Let $\epsilon>0$. Then, there is a constant $q_0(n,\epsilon)$ such that for every prime power $q\geq q_0(n,\epsilon)$, every non-zero eigenvalue $\lambda$ of $\lap{k}{n}{q}$ satisfies $
\lambda \geq n-k-2-\epsilon.
$
\end{theorem}
For the proof of Theorem \ref{thm:garland}, Garland developed his well known ``local to global" method, which relates the Laplacian eigenvalues of a simplicial complex to the eigenvalues of its links (see e.g. \cite{ballmann1997l2}).

In \cite{papikian2008eigenvalues,papikian2016garland} Papikian continued the study of the spectrum of $\lap{k}{n}{q}$. Based on Garland's results and on some computer assisted calculations, he proposed the following conjecture:

\begin{conjecture}[Papikian \cite{papikian2016garland}]\label{conj:papikian}
Let $n\geq 3$ and let $q\geq 2$ be a prime power. Then, for all $0\leq k \leq n-3$,
\begin{enumerate}[leftmargin=*]
    \item The number of distinct eigenvalues of $\lap{k}{n}{q}$ does not depend on $q$.
    \item For every $\epsilon>0$ there exists an integer $q_0(n,\epsilon)$ such that, for $q\geq q_0(n,\epsilon)$, for every non-zero eigenvalue $\lambda$ of $\lap{k}{n}{q}$ there is some $m\in \{n-k-2,n-k-1,\ldots,n-1\}$ such that
    \[
        |\lambda-m|< \epsilon.
    \]
\end{enumerate}
\end{conjecture}

Here, we focus on the $0$-dimensional case. Our main result is the following precise description of the asymptotic behavior (as $q\to\infty$) of the spectrum of $\lap{0}{n}{q}$. 

\begin{theorem}\label{thm:main_result}
Let $n\geq 3$ and let $q$ be a prime power. Then, $0$ is an eigenvalue of $\lap{0}{n}{q}$ with multiplicity $1$, and $n-1$ is an eigenvalue with multiplicity $n-2$. If $n$ is even, $n-2$ is an eigenvalue with multiplicity $\binom{n}{n/2}_q-\binom{n}{n/2-1}_q$. In addition, there exist $q_0(n),C(n)>0$ such that the following holds:

Let $1\leq k\leq \lfrac{n-1}{2}$. Let $\alpha(k)=\min\{k/2,1\}$, and
\[
    \mathcal{J}_{k}=\left\{ \pm 2\cos\left(\frac{j\pi}{n-2k+2}\right) :\, 1\leq j\leq \lfrac{n-2k+1}{2}\right\}.
\]
For every $\zeta\in \mathcal{J}_k$, define
\[
    D_{n,k,\zeta}(q)=\left\{\lambda\in \mathbb{R}:\, \left|\lambda-\left( n-2 + \zeta\cdot q^{-\frac{k}{2}}\right)\right|\leq C(n)\cdot q^{-(\frac{k}{2}+\alpha(k))}\right\}.
\]
If $n$ is odd, let $\mathcal{D}_{k}=\{D_{n,k,\zeta}(q) :\, \zeta\in\mathcal{J}_k\}$. If $n$ is even, let
\[
    \tilde{D}_{n,k}(q)=
    \left\{\lambda\in \mathbb{R}:\, \left|\lambda-\left(n-2+\frac{2(n-2k)}{n-2k+2}\cdot q^{-k}\right)\right|\leq C(n)\cdot q^{-(k+1)}\right\},
\]
and let $\mathcal{D}_k=\{\tilde{D}_{n,k}(q)\}\cup \{D_{n,k,\zeta}(q) :\, \zeta\in\mathcal{J}_k\}$. 

Then, for every prime power $q\geq q_0(n)$, the intervals in $\bigcup_{k=1}^{\lfloor (n-1)/2 \rfloor}\mathcal{D}_k$ are pairwise disjoint, and for all $1\leq k\leq \lfrac{n-1}{2}$, every interval in $\mathcal{D}_k$ contains exactly one eigenvalue of $\lap{0}{n}{q}$, with multiplicity $\binom{n}{k}_q-\binom{n}{k-1}_q$.
\end{theorem}

A simple counting argument (see the proof of Theorem \ref{thm:main_result} in Section \ref{sec:spectra}) shows that the total number of eigenvalues considered by Theorem \ref{thm:main_result} is exactly $\sum_{k=1}^{n-1}\binom{n}{k}_q=|\flagcomplex{n}{q}(0)|$. Therefore, Theorem \ref{thm:main_result} indeed determines the asymptotic behavior of the complete spectrum of $\lap{0}{n}{q}$.

As a simple consequence of Theorem \ref{thm:main_result}, we obtain a proof of the $0$-dimensional case of Papikian's conjecture for large values of $q$:

\begin{corollary}\label{cor:pap0case}
Let $n\geq 3$. Then,
\begin{enumerate}
    \item There is an integer $q_0(n)$ such that, for every prime power $q\geq q_0(n)$, the number of distinct eigenvalues of $\lap{0}{n}{q}$ is exactly $\lfrac{n^2}{4}+2$.
    \item For every $\epsilon>0$ there is an integer $q_0(n,\epsilon)$ such that, for every prime power $q\geq q_0(n,\epsilon)$, every eigenvalue $\lambda\neq 0,n-1$ of $\lap{0}{n}{q}$ satisfies
    $
        |\lambda-(n-2)|<\epsilon.
    $
\end{enumerate}

\end{corollary}

The proof of Theorem \ref{thm:main_result} consists of two main steps. First, we 
find a basis of $C^0(\flagcomplex{n}{q})$ in which we have a simple explicit description of the matrix representation of the Laplacian $\lap{0}{n}{q}$:
\begin{theorem}\label{thm:diagonal_representation}
There is a basis $B$ of $C^0(\flagcomplex{n}{q})$ such that the matrix representation of $\lap{0}{n}{q}$ with respect to the basis $B$ is a block diagonal matrix
\[
    \begin{pmatrix}
    L_0 &   & \\
     & \ddots & \\
     & & L_{\lfrac{n}{2}}
\end{pmatrix}
\]
with blocks
\[
    L_k=  I_{\binom{n}{k}_q-\binom{n}{k-1}_q}\otimes \tilde{L}_k,
\]
where $I_{\binom{n}{k}_q-\binom{n}{k-1}_q}$ is the $\left(\binom{n}{k}_q-\binom{n}{k-1}_q\right)\times\left(\binom{n}{k}_q-\binom{n}{k-1}_q\right)$ identity matrix, $\tilde{L}_0$ is the $(n-1)\times (n-1)$ matrix with entries
\begin{equation}\label{eq:lap_block_formula_k0}
    (\tilde{L}_0)_{ij}=\begin{cases}
    n-2 & \text{ if } i=j,\\
    -1 &\text{ if } i\neq j
    \end{cases}
\end{equation}
for $1\leq i,j\leq n-1$, and, for $1\leq k\leq \lfrac{n}{2}$, $\tilde{L}_k$ is the $(n-2k+1)\times (n-2k+1)$ matrix with entries
\begin{equation}\label{eq:lap_block_formula}
    (\tilde{L}_k)_{ij}=\begin{cases}
    n-2 & \text{ if } i=j,\\
    -c_{i j k} \binom{n-i}{j-i}_q^{-1} &\text{ if } i< j,\\
    - \binom{i-k}{j-k}_q \binom{i}{j}_q^{-1} & \text{ if } i> j
    \end{cases}
\end{equation}
for $k\leq i,j\leq n-k$, where 
\[
c_{i j k}= \sum_{r=0}^k (-1)^{k-r} q^{\binom{r+1}{2}+\binom{k}{2}-rk}
    \binom{k}{r}_q \binom{n-i-k+r}{j-i-k+r}_q.
\]
\end{theorem}
\begin{remark}[Remarks.] \
\begin{enumerate}
    \item The symbol $\otimes$ in Theorem \ref{thm:diagonal_representation} denotes the Kronecker product. Recall that for matrices $A\in \Rea^{n\times n}$ and $B\in \Rea^{m\times m}$, $A\otimes B$ is the $nm\times nm$ matrix obtained from $A$ by replacing each element $a_{ij}$ with the $m\times m$ block $a_{ij} B$ (so that $I_{\binom{n}{k}_q-\binom{n}{k-1}_q}\otimes \tilde{L}_k$ is just a block diagonal matrix formed by $\binom{n}{k}_q-\binom{n}{k-1}_q$ copies of $\tilde{L}_k$).
   \item  Note that, for $k\geq 1$, the rows and columns of the matrices $\tilde{L}_k$ in Theorem \ref{thm:diagonal_representation} are indexed from $k$ to $n-k$ (and not from $1$ to $n-2k+1$).
\end{enumerate}
\end{remark}
One of the main ``advantages" in the choice of the basis $B$ in Theorem \ref{thm:diagonal_representation} is that the size of the blocks $\tilde{L}_k$ does not depend on $q$, which facilitates the study of the spectrum of $\lap{0}{n}{q}$ as $q$ grows. The proof of Theorem \ref{thm:diagonal_representation} relies on the study of subspace inclusion matrices, introduced by Kantor in \cite{kantor1972incidence}.

The second step in the proof of Theorem \ref{thm:main_result} consists of estimating the eigenvalues of the matrices $\tilde{L}_k$ defined in Theorem \ref{thm:diagonal_representation}. In order to do this, we establish a relation between
the characteristic polynomial of $\tilde{L}_k$ and the sign-alternating Fibonacci polynomials studied by Donnelly, Dunkum, Huber and Knupp in \cite{donnellysign}.

This paper is organized as follows. In Section \ref{sec:prelims} we give some background on Laplacian operators on simplicial complexes and on $q$-binomial coefficients. In Section \ref{sec:weight} we present some simple results on the weight function $w$ and use them to prove Lemma \ref{lemma:laplacian_formula} and Corollary \ref{cor:laplacian_formula_0_dim}.
In Section \ref{sec:lap_inclusion_matrices} we introduce the ``subspace inclusion matrices" $A_{ij}$, we study some of their properties and explain their relation to the Laplacian matrix $\lap{0}{n}{q}$. In Section \ref{sec:proof_diag_form} we present a proof of Theorem \ref{thm:diagonal_representation}. Section \ref{sec:spectra} deals with the last step in the proof of Theorem \ref{thm:main_result}, the evaluation of the spectra of the matrices $\tilde{L}_k$. In the last section we make some remarks on possible directions for future research.

\section{Preliminaries}\label{sec:prelims}

\subsection{Weighted Laplacians}\label{sec:prelims_lap}

A \emph{simplicial complex} $X$ is a family of subsets of some finite set $V$, such that if $\sigma\in X$ and $\sigma'\subset \sigma$, then $\sigma'\in X$. The set $V$ is called the \emph{vertex set} of $X$. The elements of $X$ are called the \emph{simplices} or \emph{faces} of the complex. The \emph{dimension} of a simplex $\sigma\in X$ is defined as $|\sigma|-1$.

Let $X$ be a simplicial complex on vertex set $V$, and let $k\geq -1$. 
An \emph{ordered $k$-simplex} $[v_0,\ldots,v_k]$ is a $k$-dimensional simplex $\{v_0,\ldots,v_k\}\in X$ together with an order of its vertices. For ordered simplices $\sigma,\tau$ such that $\sigma=\tau$ as sets, let $(\sigma:\tau)$ be the sign of the permutation mapping $\sigma$ to $\tau$. Similarly, for ordered simplices $\sigma,\tau$ in $X$ such that $\tau\subset \sigma$ and $\sigma\setminus\tau= \{v\}$ for some vertex $v\in \sigma$, let $(\sigma:\tau)$ be the sign of the permutation mapping $\sigma$ to $v\tau$ (where $v\tau$ is the ordering of $\tau\cup\{v\}$ where $v$ is the first vertex, followed by the vertices of $\tau$ in their original order).

A \emph{$k$-cochain} is an $\mathbb{R}$-valued skew-symmetric function on the ordered $k$-simplices of $X$. That is, $\phi$ is a $k$-cochain if for any two ordered $k$-simplices $\sigma,\tilde{\sigma}$ in $X$ that are equal as sets, it satisfies $\phi(\tilde{\sigma})=(\sigma:\tilde{\sigma}) \phi(\sigma)$.

Let $C^k(X)$ denote the space of $k$-cochains on $X$. The \emph{coboundary operator} $d_k: C^k(X)\to C^{k+1}(X)$ is defined by
\[
    d_k\phi([v_0,\ldots,v_{k+1}]) = \sum_{i=0}^{k+1} (-1)^i \phi([v_0,\ldots,v_{i-1},v_{i+1},\ldots,v_{k+1}])
\]
for every $k$-cochain $\phi$ and every ordered $(k+1)$-simplex $[v_0,\ldots,v_{k+1}]$.

Let $X(k)$ be the set of all $k$-simplices of $X$, each given a fixed arbitrary order.
Let $w: X \to \Rea^{+}$ be a weight function on the simplices of $X$. We define an inner product on the vector space $C^k(X)$ by
\[
    \langle \phi,\psi \rangle =\sum_{\sigma\in X(k)} w(\sigma)\phi(\sigma)\psi(\sigma).
\]
Let $d_k^*$ be the operator adjoint to $d_k$ with respect to this inner product. We define the \emph{weighted upper $k$-Laplacian} $\Delta_k^{+}(X): C^k(X)\to C^k(X)$ by $\Delta_k^+(X)= d_k^* d_k.$

Note that, by definition, $\Delta_k^{+}(X)$ is a positive semi-definite operator. In particular, it is diagonalizable, and all its eigenvalues are real and non-negative.

Let $k\geq 0$ and $\sigma\in X(k)$. We define the $k$-cochain $1_{\sigma}$ by
\[
1_{\sigma}(\tau)=\begin{cases}
(\sigma: \tau) & \text{ if $\sigma=\tau$ (as sets)},\\
0 & \text{ otherwise.}
\end{cases}
\]

The set $\{1_{\sigma}\}_{\sigma\in X(k)}$ forms a basis of the space $C^k(X)$, that we will call the standard basis.
We will identify the operator $\Delta_k^{+}(X)$ with its matrix representation in the standard basis. For $\sigma,\tau\in X(k)$, we will denote by $\Delta_k^{+}(X)_{\sigma,\tau}$ the matrix element at row indexed by $1_\sigma$ and column indexed by $1_\tau$. We say that $v\in \lk(X,\sigma)$ if $v\in V\setminus \sigma$ and $\sigma\cup\{v\}\in X$.

\begin{lemma}[See e.g. {\cite{horak2013spectra}}]
\label{lemma:laplacian_matrix_rep}
Let $\sigma,\tau\in X(k)$. Then, 
\[
    \Delta_k^{+}(X)_{\sigma,\tau}=\begin{cases}
        \sum_{v\in \lk(X,\sigma)} \frac{w(\sigma\cup\{v\})}{w(\sigma)} & \text{ if } \sigma=\tau,\\
        -\frac{w(\sigma\cup\tau)}{w(\sigma)}(\sigma:\sigma\cap\tau)(\tau:\sigma\cap\tau) & \text{ if } |\sigma\cap\tau|=k,\, \sigma\cup \tau\in X(k+1),\\
        0 & \text{ otherwise.}
    \end{cases}
\]
\end{lemma}

\subsection{$q$-Binomial coefficients}\label{sec:bin_prelims}

Let $q$ be a prime power, and let $a,b$ be integers. The \emph{$q$-binomial coefficient} $\binom{a}{b}_q$ is the number of $b$-dimensional subspaces contained in $\FF_q^a$. More explicitly, we have
\begin{equation}\label{eq:binom}
    \binom{a}{b}_q=\frac{\prod_{i=1}^a (q^i-1)}{\prod_{i=1}^b (q^i-1) \prod_{i=1}^{a-b}(q^i-1)}= \frac{\prod_{i=a-b+1}^a (q^i-1)}{\prod_{i=1}^{b}(q^i-1)}
\end{equation}
if $a\geq b\geq 0$, and $\binom{a}{b}_q=0$ otherwise. 

We will need the following well known basic result on $q$-binomial coefficients:
\begin{lemma}\label{lemma:countsubspaces}
Let $0\leq r\leq k\leq n$. Let $U$ be an $r$-dimensional subspace of $\FF_q^n$. Then, the number of $k$-dimensional subspaces of $\FF_q^n$ that contain $U$ is $\binom{n-r}{k-r}_q$.
\end{lemma}
\begin{proof}
The number of $k$-dimensional subspaces of $\FF_q^n$ containing $U$ is equal to the number of $(k-r)$-dimensional subspaces of $\FF_q^n/U\cong \FF_q^{n-r}$. Hence, there are $\binom{n-r}{k-r}_q$ such subspaces.
\end{proof}

We will need the following simple results about the behavior of the $q$-binomial coefficients as $q$ tends to infinity.
We will use the following asymptotic notation: Let $f,g:\mathbb{R}\to\mathbb{R}$. We say that $f(x)=\bigoh{g(x)}$ if there exist $x_0, C>0$ such that $|f(x)|\leq C g(x)$ for all $x\geq x_0$.

\begin{lemma}\label{lemma:limit_of_qbinoms}
Let $0\leq  b\leq a$, and let $k\leq b$. Then,
\[
  \binom {a-k}{b-k}_q \binom{a}{b}_q^{-1}= q^{-k(a-b)}\left(1+O(q^{-1})\right).
\]
\end{lemma}
\begin{proof}
By \eqref{eq:binom}, we can write $\binom{a-k}{b-k}_q= f(q)/g(q)$, where $f$ is a monic polynomial in $q$ of degree $\sum_{i=a-b+1}^{a-k} i= \frac{1}{2}(b-k)(2a-b-k+1)$ and $g$ is a monic polynomial in $q$ of degree $\sum_{i=1}^{b-k} i=\frac{1}{2}(b-k)(b-k+1)$. Similarly, we can write $\binom{a}{b}_q=h(q)/s(q)$, where $h$ is a monic polynomial in $q$ of degree $\frac{1}{2}b(2a-b+1)$ and $s$ is a monic polynomial in $q$ of degree $\frac{1}{2}b(b+1)$.
Therefore, if $k\geq 0$, we can write
\begin{align*}
\frac{\binom{a-k}{b-k}_q \cdot q^{k(a-b)}}{\binom{a}{b}_q}-1 &= 
\frac{\frac{f(q)}{g(q)}\cdot q^{k(a-b)}}{\frac{h(q)}{s(q)}} -1 
=\frac{f(q)\cdot s(q)\cdot q^{k(a-b)}}{h(q)\cdot g(q)}-1
\\
& = \frac{f(q)\cdot s(q)\cdot q^{k(a-b)} - h(q)\cdot g(q)}{h(q)\cdot g(q)}
\end{align*}
Note that
\[
\deg(f(q)\cdot s(q)\cdot q^{k(a-b)})=\frac{1}{2} (2ab+2b-2bk+k^2-k) = \deg(h(q)\cdot g(q)). 
\]
Since $f(q)\cdot s(q)\cdot q^{k(a-b)}$ and $h(q)\cdot g(q)$ are both monic polynomials, the degree of the polynomial $f(q)\cdot s(q)\cdot q^{k(a-b)}-h(q)\cdot g(q)$ is strictly smaller than that of $h(q)\cdot g(q)$. So, we obtain
\[
\frac{\binom{a-k}{b-k}_q \cdot q^{k(a-b)}}{\binom{a}{b}_q}-1= O(q^{-1}).
\]
Similarly, if $k<0$, we can write
\[
\frac{\binom{a-k}{b-k}_q \cdot q^{k(a-b)}}{\binom{a}{b}_q}-1 = \frac{\binom{a-k}{b-k}_q}{\binom{a}{b}_q\cdot q^{-k(a-b)}}-1
= \frac{f(q)\cdot s(q) - h(q)\cdot g(q) \cdot q^{-k(a-b)}}{h(q)\cdot g(q)\cdot q^{-k(a-b)}}.
\]
Note that
\[
\deg(f(q)\cdot s(q))=\deg(h(q)\cdot g(q)\cdot q^{-k(a-b)}). 
\]
Since $f(q)\cdot s(q)$ and $h(q)\cdot g(q)\cdot q^{-k(a-b)}$ are both monic polynomials, the degree of the polynomial $f(q)\cdot s(q) -h(q)\cdot  g(q)\cdot q^{-k(a-b)}$ is strictly smaller than that of $h(q)\cdot g(q)\cdot q^{-k(a-b)}$. So, we obtain
\[
\frac{\binom{a-k}{b-k}_q \cdot q^{k(a-b)}}{\binom{a}{b}_q}-1= O(q^{-1}),
\]
as wanted.
\end{proof}

As an immediate consequence, we obtain
\begin{lemma}\label{lemma:qbinom_asymptotics}
Let $a\geq b\geq 0$ be integers. Then,
\[
\binom{a}{b}_q = q^{b(a-b)}\left(1+O(q^{-1})\right).
\]
\end{lemma}
\begin{proof}
We can write
\[
\binom{a}{b}_q= \binom{(a-b)-(-b)}{0-(-b)}_q \binom{a-b}{0}_q^{-1}. 
\]
So, by Lemma \ref{lemma:limit_of_qbinoms}, we obtain
\[
\binom{a}{b}_q = q^{b(a-b)}\left(1+O(q^{-1})\right).
\]
\end{proof}

We will also use the following inversion formula due to Carlitz:

\begin{lemma}[Carlitz \cite{carlitz1973some}]\label{lemma:carlitz}
Let $\{a_n\}_{n=0}^m$ and $\{b_n\}_{n=0}^m$ be two sequences such that
\[
    a_n= \sum_{k=0}^{n} \binom{n}{k}_q b_k
\]
for all $0\leq n\leq m$. Then,
\[
    b_n= \sum_{k=0}^{n} (-1)^{n-k} q^{\binom{k+1}{2}+\binom{n}{2}-kn} \binom{n}{k}_q a_k
\]
for all $0\leq n\leq m$.
\end{lemma}

\section{The weight function}\label{sec:weight}

Let $V_1\subsetneq V_2\subsetneq \cdots \subsetneq V_k$ be a flag in $\FF_q^n$, and let $\sigma=\{V_1,\ldots,V_k\}\in\flagcomplex{n}{q}$. Recall that we defined the weight function $w(\sigma)$ to be the number of complete flags extending $\sigma$ or, equivalently, the number of maximal faces of $\flagcomplex{n}{q}$ containing $\sigma$.
In this section we discuss some useful properties of the weight function $w$.

Let $\gl{n,q}$ be the group of invertible $n\times n$ matrices over $\FF_q$. For $g\in \gl{n,q}$ and a subspace $V\subset \FF_q^n$, let
$
    gV=\{gv :\, v\in V\}.
$
Note that $gV$ is also a subspace of $\FF_q^n$ and that $\dim(gV)=\dim(V)$.

Let $F$ be the flag $V_1\subsetneq V_2\subsetneq \cdots \subsetneq V_k$. Then, we denote by $gF$ the flag
\[
    gV_1\subsetneq gV_2\subsetneq \cdots \subsetneq gV_k.
\]

First, we show that $w$ depends only on the dimensions of the subspaces forming the flag:

\begin{lemma}\label{lemma:weight_function_depends_on_dim}
Let $V_1\subsetneq V_2\subsetneq\cdots\subsetneq V_k$ and $W_1\subsetneq W_2\subsetneq\cdots\subsetneq W_k$ be two flags in $\FF_q^n$. Let $\sigma=\{V_1,\ldots,V_k\}$ and $\sigma'=\{W_1,\ldots,W_k\}$. If $\dim(V_i)=\dim(W_i)$ for all $1\leq i\leq k$, then $w(\sigma)=w(\sigma')$.
\end{lemma}
\begin{proof}
Let $v_1,\ldots,v_n$ be a basis of $\FF_q^n$ such that, for all $1\leq i\leq k$, the first $\dim(V_i)$ vectors in the basis form a basis for $V_i$. Similarly, let $w_1,\ldots,w_n$ be a basis of $\FF_q^n$ such that, for all $1\leq i\leq k$, the first $\dim(W_i)=\dim(V_i)$ vectors in the basis form a basis for $W_i$.

Let $g\in \gl{n,q}$ be the linear isomorphism that maps $v_i$ to $w_i$ for all $1\leq i\leq n$. Then, for all $1\leq i\leq k$, we have $gV_i=W_i$.

Let $\mathcal{F}$ be the set of complete flags extending $\sigma$, and let $\mathcal{F}'$ be the set of complete flags extending $\sigma'$. We have a map $\tilde{g}: \mathcal{F}\to \mathcal{F}'$ defined by $\tilde{g}(F)=gF$ for all $F\in\mathcal{F}$. Note that $\tilde{g}$ is a bijection: indeed, it is easy to check that $\widetilde{g^{-1}}:\mathcal{F}'\to\mathcal{F}$ is its inverse. Hence, $w(\sigma)=|\mathcal{F}|=|\mathcal{F}'|=w(\sigma')$.
\end{proof}

\begin{lemma}\label{lemma:weightsum}
Let $k$ be an integer. Let $V_1\subsetneq \cdots\subsetneq V_{i-1}\subsetneq V_{i+1}\subsetneq\cdots \subsetneq V_k$ be a flag in $\FF_q^n$, and let $\tau=\{V_1,\ldots,V_{i-1},V_{i+1},\ldots,V_k\}$. Let $\dim(V_{i-1})<d<\dim(V_{i+1})$.
Let $\mathcal{U}$ be the set of $d$-dimensional subspaces $U\subset \FF_q^n$ satisfying $V_{i-1}\subset U\subset V_{i+1}$. For $U\in\mathcal{U}$, let $\sigma_U= \{V_1,\ldots,V_{i-1},U,V_{i+1},\ldots,V_k\}$. Then,
\[
    w(\tau)= \sum_{U\in\mathcal{U}} w(\sigma_{U}).
\]
\end{lemma}
\begin{proof}
Notice that, for all  $U\in\mathcal{U}$, any complete flag extending $\sigma_U$ extends also $\tau$. Moreover, each complete flag extending $\tau$ extends exactly one of the flags $\sigma_U$. Therefore, we have
\[
    w(\tau)= \sum_{U\in\mathcal{U}} w(\sigma_U)
\]
as wanted.
\end{proof}

Let $V_1\subsetneq \cdots \subsetneq V_k$ be a flag in $\FF_q^n$. Let $\sigma=\{V_1,\ldots,V_k\}$ and $\eta=\sigma\setminus\{V_i\}$. Recall that we defined $r(\sigma,\eta)=\dim(V_{i+1})-\dim(V_{i-1})$ and $t(\sigma,\eta)=\dim(V_{i})-\dim(V_{i-1})$ (where $V_0=\{0\}$ and $V_{k+1}=\FF_q^n$).

\begin{lemma}\label{lemma:weights1}
Let $V_1\subsetneq \cdots \subsetneq V_k$ be a flag in $\FF_q^n$. Let $\sigma=\{V_1,\ldots,V_k\}$ and $\eta=\sigma\setminus\{V_i\}$. Then
\[
    \frac{w(\sigma)}{w(\eta)}=\binom{r(\sigma,\eta)}{t(\sigma,\eta)}_q^{-1}.
\]

\end{lemma}
\begin{proof}

Let $r=r(\sigma,\eta)$ and $t=t(\sigma,\eta)$.
Let $\mathcal{U}$ be the family of subspaces $U\subset \FF_q^n$ satisfying $V_{i-1}\subset U \subset V_{i+1}$ and $\dim(U)=\dim(V_i)$. By Lemma \ref{lemma:countsubspaces}, we have $|\mathcal{U}|=\binom{r}{t}_q$.

For every $U\in \mathcal{U}$, let $\sigma_U=\{V_1,\ldots,V_{i-1},U,V_{i+1},\ldots,V_k\}$. Then, we have $\sigma=\sigma_{V_i}$, and, by Lemma \ref{lemma:weight_function_depends_on_dim}, $w(\sigma_U)=w(\sigma)$ for all $U\in\mathcal{U}$.

Therefore, by Lemma \ref{lemma:weightsum}, we have
\[
    w(\eta)= \sum_{U\in\mathcal{U}} w(\sigma_U)= |\mathcal{U}|\cdot w(\sigma)= \binom{r}{t}_q \cdot w(\sigma).
\]
We obtain $\frac{w(\sigma)}{w(\eta)}=\binom{r}{t}_q^{-1}$, as wanted.
\end{proof}

\begin{lemma}\label{lemma:diagonal}
Let $\sigma\in\flagcomplex{n}{q}(k)$. Then
\[
    \lap{k}{n}{q}_{\sigma,\sigma}=n-k-2.
\]
\end{lemma}
\begin{proof}

Let $\sigma=\{V_1,\ldots,V_{k+1}\}$, and let $d_i=\dim(V_i)$ for all $1\leq i\leq k+1$.
By Lemma \ref{lemma:laplacian_matrix_rep}, we have
\begin{align*}
\lap{k}{n}{q}_{\sigma,\sigma}&=\frac{1}{w(\sigma)} \sum_{U\in \lk(\flagcomplex{n}{q},\sigma)} w(\sigma\cup\{U\})
\\
&=\frac{1}{w(\sigma)} \sum_{\substack{d\in[n-1],\\d\notin\{d_1,\ldots,d_{k+1}\}}} \sum_{\substack{U\in \lk(\flagcomplex{n}{q},\sigma)\\
\dim(U)=d}} w(\sigma\cup\{U\}).
\end{align*}
By Lemma \ref{lemma:weightsum}, we have for all $d\in[n-1]\setminus \{d_1,\ldots,d_{k+1}\}$:
\[
\sum_{\substack{U\in \lk(\flagcomplex{n}{q},\sigma)\\
\dim(U)=d}} w(\sigma\cup\{U\})= w(\sigma).
\]
So, we obtain
\begin{align*}
\lap{k}{n}{q}_{\sigma,\sigma}&=
\frac{1}{w(\sigma)} \sum_{\substack{d\in[n-1],\\ d\notin\{d_1,\ldots,d_{k+1}\}}} \sum_{\substack{U\in \lk(\flagcomplex{n}{q},\sigma)\\
\dim(U)=d}} w(\sigma\cup\{U\})
\\
&= \frac{1}{w(\sigma)} \sum_{\substack{d\in[n-1],\\ d\notin\{d_1,\ldots,d_{k+1}\}}} w(\sigma) = n-k-2.
\end{align*}
\end{proof}

\begin{remark}
    An explicit formula for the weight function $w(\sigma)$ is given in \cite[Lemma 5.4]{papikian2016garland}, from which it is possible to recover Lemma \ref{lemma:weights1} and Lemma \ref{lemma:diagonal}.
\end{remark}

\begin{proof}[Proof of Lemma \ref{lemma:laplacian_formula}]

For $k\geq0$, we will consider each flag in $\flagcomplex{n}{q}(k)$ as an ordered simplex, ordered by increasing dimension. That is, for $\sigma=\{V_1,\ldots,V_k\}\in \flagcomplex{n}{q}(k)$, where $V_1\subset \cdots \subset V_{k+1}$, we  identify $\sigma$ with the ordered simplex $[V_1,\ldots,V_{k+1}]$.
For any $1\leq i\leq k+1$, let $\sigma_i\in\flagcomplex{n}{q}(k-1)$ be obtained from $\sigma$ by removing $V_i$. Then $(\sigma:\sigma_i)$ is $(-1)^{i-1}$. So, for $1\leq i<j\leq k+1$, we have $(\sigma:\sigma_i)(\sigma:\sigma_j)=(-1)^{i-1}(-1)^{j-1}=(-1)^{j-i}=(-1)^{\epsilon(\sigma_i,\sigma_j)}$. So, by Lemma \ref{lemma:laplacian_matrix_rep}, Lemma \ref{lemma:weights1} and Lemma \ref{lemma:diagonal}, we obtain
    \[
    \lap{k}{n}{q}_{\sigma,\tau}=\begin{cases}
                    n-k-2 & \text{if } \sigma=\tau,\\
                    (-1)^{\epsilon(\sigma,\tau)+1}\binom{r(\sigma\cup\tau,\sigma)}{t(\sigma\cup\tau,\sigma)}_q^{-1} & \text{if } |\sigma\cap \tau|=k,\, \sigma\cup\tau\in\flagcomplex{n}{q},\\
                    0 & \text{otherwise,}
    \end{cases}
    \]
    for all $\sigma,\tau\in \flagcomplex{n}{q}(k)$.
\end{proof}

\begin{proof}[Proof of Corollary \ref{cor:laplacian_formula_0_dim}]
Let $U, V$ be two non-trivial subspaces of $\FF_q^n$. Assume $U\neq V$. Note that $\{U,V\}\in\flagcomplex{n}{q}$ if and only of $U\subsetneq V$ or $V\subsetneq U$, and in this case we have $\epsilon(\{U\},\{V\})=0$. Assume $\dim(U)=i$ and $\dim(V)=j$. Then, if $U\subsetneq V$, we have $r(\{U,V\},\{U\})=n-i$ and $t(\{U,V\},\{U\})=j-i$, and if $V\subsetneq U$ we have $r(\{U,V\},\{U\})=i$ and $t(\{U,V\},\{U\})=j$. Therefore, by Lemma \ref{lemma:laplacian_formula}, we obtain
   \[
    \lap{0}{n}{q}_{U,V}=\begin{cases}
                    n-2 & \text{if } U=V,\\
                    - \binom{n-i}{j-i}_q^{-1} & \text{if } U\subsetneq V,\\
                    - \binom{i}{j}_q^{-1} & \text{if } V\subsetneq U,\\
                    0 & \text{otherwise.}
    \end{cases}
    \]
\end{proof}

\section{Subspace inclusion matrices}\label{sec:lap_inclusion_matrices}

For $0\leq i\leq n$, denote by $S(i)$ the collection of all subspaces of $\mathbb{F}_q^n$ of dimension $i$.

Let $1\leq i\leq n-1$ and $U\in S(i)$. Define the cochain $1_U\in C^0(\flagcomplex{n}{q})$ as
\[
   1_U(V)=\begin{cases}
            1 & \text{ if } U=V,\\
            0 & \text{ otherwise.}
            \end{cases}
\]
Recall that $\cup_{i=1}^{n-1}\{1_U:\, U\in S(i)\}$ forms a basis for $C^{0}(\flagcomplex{n}{q})$.

Let $0\leq i,j\leq n$. Let $A_{i j}$ be the $S(i)\times S(j)$ matrix
\[
    (A_{ij})_{U,V}=\begin{cases}
    1 & \text{ if } U\subset V \text{ or } V\subset U,\\
    0 & \text{ otherwise.}
    \end{cases}
\]
Note that $A_{ij}=A_{ji}^t$, and that $A_{ii}$ is just the $S(i)\times S(i)$ identity matrix $I$. Also, for all $0\leq j\leq n$, $A_{0j}\in \mathbb{R}^{1\times S(j)}=\mathbb{R}^{S(j)}$ is the all-ones vector.
Using the matrices $A_{i j}$, we can give the following description for $\lap{0}{n}{q}$, which follows immediately from Corollary \ref{cor:laplacian_formula_0_dim}:

\begin{proposition}\label{lemma:matrix_form_L0}
We can write $\lap{0}{n}{q}$ as an $(n-1)\times (n-1)$ block matrix
\[
    \lap{0}{n}{q}= \left( \begin{array}{c c c}
    L_{1,1} & \cdots & L_{1,n-1}\\
    \vdots & & \vdots \\ 
    L_{n-1,1} & \cdots & L_{n-1,n-1}
    \end{array}\right),
\]
where, for $1\leq i,j\leq n-1$, $L_{i j}$ is the $S(i)\times S(j)$ matrix
\[
    L_{ij}=\begin{cases}
    (n-2) I & \text{ if } i=j,\\
    -\binom{n-i}{j-i}_q^{-1} A_{ij} & \text{ if } i<j,\\
    -\binom{i}{j}_q^{-1} A_{ij} & \text{ if } i>j.
    \end{cases}
\]
\end{proposition}

We will need the following results regarding products of subspace inclusion matrices. First, we will need the following lemma of Kantor (see also \cite{gottlieb1966certain,graver1973module,wilson1973necessary} for analogous results in the case $q=1$, that is, for ``subset inclusion" matrices).

\begin{lemma}[Kantor \cite{kantor1972incidence}]\label{lemma:product1}
Let $k\leq j\leq i$. Then
\[
    A_{ij}A_{jk} = \binom{i-k}{j-k}_q A_{ik}.
\]
\end{lemma}

\begin{proof}
Let $U\in S(k)$. Then,
\[
    A_{i j} A_{j k} 1_U
= A_{i j} \left( \sum_{\substack{V\in S(j):\\ U\subset V}} 1_V\right)   
    = \sum_{\substack{V\in S(j):\\ U\subset V}} A_{i j} 1_V
    = \sum_{\substack{V\in S(j): \\ U\subset V}} \sum_{\substack{W\in S(i): \\ V\subset W}}1_W.
\]
Let $W\in S(i)$ such that $U\subset W$. By Lemma \ref{lemma:countsubspaces}, the number of $j$-dimensional subspaces of $W$ containing $U$ is $\binom{i-k}{j-k}_q$. So, we obtain
\[
   A_{i j} A_{j k} 1_U = \sum_{\substack{W\in S(i): \\ U\subset W}} \binom{i-k}{j-k}_q 1_W
    =\binom{i-k}{j-k}_q A_{i k} 1_U.
\]
Thus, $A_{ij}A_{jk} = \binom{i-k}{j-k}_q A_{ik}$.
\end{proof}

\begin{lemma}\label{lemma:product2_preparation}
Let $0\leq k\leq i$, and let $U\in S(k)$. Then, for every sequence $\{\alpha_r\}_{r=0}^k$, we can write
\begin{equation}\label{eq:subspacesum}
\sum_{r=0}^k \alpha_r \sum_{\substack{V\in S(r):\\ V\subset U}} \sum_{\substack{W\in S(i):\\  V\subset W}} 1_W =
\sum_{m=0}^k \sum_{\substack{W\in S(i):\\ \dim(U\cap W)=m}} \beta_m 1_W,
\end{equation}
where
\[
    \beta_m=\sum_{r=0}^{m} \binom{m}{r}_q \alpha_r
\]
for all $0\leq m\leq k$.
\end{lemma}
\begin{proof}
Let $0\leq m\leq k$, and let $W\in S(i)$ such that $\dim(U\cap W)=m$. The coefficient of $1_W$ on the right-hand side of Equation \eqref{eq:subspacesum} is $\beta_m$. 

Let $0\leq r\leq k$. The number of $r$-dimensional subspaces of $U\cap W$ is $\binom{m}{r}_q$ if $r\leq m$ and $0$ otherwise. Therefore, the coefficient of $1_W$ on the left-hand side of Equation \eqref{eq:subspacesum} is
\[
    \sum_{r=0}^m \binom{m}{r}_q \alpha_r.
\]
We obtain
\[
    \beta_m=\sum_{r=0}^{m} \binom{m}{r}_q \alpha_r
\]
for all $0\leq m\leq k$, as wanted.
\end{proof}

\begin{lemma}\label{lemma:product2}
Let $0\leq k\leq i\leq j\leq n$. Then,
\[
    A_{i j}A_{j k}=
    \sum_{m=0}^k c_{i j k m} A_{i m} A_{m k},
\]
where, for all $0\leq m\leq k$, 
\begin{equation}\label{eq:coeffs1}
    c_{i j k m}=\sum_{r=0}^m (-1)^{m-r} q^{\binom{r+1}{2}+\binom{m}{2}-rm} \binom{m}{r}_q \binom{n-i-k+r}{j-i-k+r}_q.
\end{equation}
\end{lemma}
\begin{proof}

 Let $U\in S(k)$. Then, we can write
\[
    A_{i j} A_{j k} 1_U = \sum_{\substack{V\in S(j): \\ U\subset V}} \sum_{\substack{W\in S(i): \\ W\subset V}} 1_W.
\]
Let $0\leq m\leq k$ and let $W\in S(i)$ such that $\dim(U\cap W)=m$. Then, by Lemma \ref{lemma:countsubspaces}, since $\dim(U+W)=\dim(U)+\dim(W)-\dim(U\cap W)= k+i-m$, the number of $j$-dimensional subspaces of $\FF_q^n$ containing both $U$ and $W$ is
\[
    \binom{n-(k+i-m)}{j-(k+i-m)}_q = \binom{n-i-k+m}{j-i-k+m}_q.
\]  
Therefore,
\[
    A_{i j} A_{j k} 1_U =\sum_{m=0}^k \sum_{\substack{W\in S(i):\\ \dim(U\cap W)=m}} \binom{n-i-k+m}{j-i-k+m}_q 1_W.
\]
Using Lemma \ref{lemma:product2_preparation} and the fact that
$\sum_{\substack{V\in S(m):\\  V\subset U}} 
  \sum_{\substack{W\in S(i):\\ V\subset W}}  1_W= A_{i m} A_{m k} 1_U$,  we obtain

\begin{align*}
    A_{i j} A_{j k} 1_U 
    = &\sum_{m=0}^k \sum_{\substack{W\in S(i):\\ \dim(U\cap W)=m}} \binom{n-i-k+m}{j-i-k+m}_q 1_W 
    \\
    =&
    \sum_{m=0}^k c_{i j k m} \sum_{\substack{V\in S(m):\\ V\subset U}} 
  \sum_{\substack{W\in S(i):\\ V\subset W}}  1_W \\
    = &
     \sum_{m=0}^k c_{i j k  m} A_{i m} A_{m k} 1_U, 
\end{align*}
where the coefficients $c_{i j k m}$ satisfy the relations
\[
   \sum_{r=0}^m \binom{m}{r}_q c_{i j k r} = \binom{n-i-k+m}{j-i-k+m}_q
\]
for all $0\leq m\leq k$. 
Finally, by Lemma \ref{lemma:carlitz}, we obtain 
\begin{equation*}
    c_{i j k m}=\sum_{r=0}^m (-1)^{m-r} q^{\binom{r+1}{2}+\binom{m}{2}-rm} \binom{m}{r}_q \binom{n-i-k+r}{j-i-k+r}_q.
\end{equation*}
\end{proof}

For convenience, we will denote $c_{i j k } = c_{i j k k}$. 
We will need the following Lemma about the asymptotic behavior of the numbers $c_{i j k}$ as $q$ tends to infinity.

\begin{lemma}\label{lemma:coeffs_limit}
Let $k\leq i< j\leq n-k$. Then, $c_{ij0}=\binom{n-i}{j-i}_q$, and for $k\geq 1$,
\begin{equation}\label{eq:cijkm_asymptotics}
c_{i j k}\binom{n-i}{j-i}_q^{-1}= 1- q^{-(n-k-j+1)}(1+O(q^{-1}))= 1 + O(q^{-1}).
\end{equation}
\end{lemma}
\begin{proof}
By \eqref{eq:coeffs1}, we have
\[
c_{ijk}= \sum_{r=0}^k (-1)^{k-r} q^{\binom{r+1}{2}+\binom{k}{2}-rk} \binom{k}{r}_q \binom{n-i-k+r}{j-i-k+r}_q.
\]
For $r=k$, we have
\[
(-1)^{k-r} q^{\binom{r+1}{2}+\binom{k}{2}-rk} \binom{k}{r}_q \binom{n-i-k+r}{j-i-k+r}_q 
= \binom{n-i}{j-i}_q.
\]
In particular, for $k=0$, we obtain
\[
    c_{ij0}= \binom{n-i}{j-i}_q.
\]
Assume $k\geq 1$. For $r<k$, we divide into two cases: if $j-i-k+r<0$, then
\[
(-1)^{k-r} q^{\binom{r+1}{2}+\binom{k}{2}-rk} \binom{k}{r}_q \binom{n-i-k+r}{j-i-k+r}_q =0.
\]
Otherwise, we have by Lemma  \ref{lemma:qbinom_asymptotics}  and Lemma \ref{lemma:limit_of_qbinoms},
\begin{multline*}
 (-1)^{k-r} q^{\binom{r+1}{2}+\binom{k}{2}-rk} \binom{k}{r}_q \binom{n-i-k+r}{j-i-k+r}_q 
\binom{n-i}{j-i}_q^{-1}
\\
 = (-1)^{k-r} q^{\binom{r+1}{2}+\binom{k}{2}-rk} \cdot q^{r(k-r)}\cdot q^{-(k-r)(n-j)} (1+O(q^{-1}))
 \\
  = (-1)^{k-r} q^{\frac{-r^2 +r (1+2(n-j)) +(k^2-k-2k(n-j))}{2}} (1+O(q^{-1})).
\end{multline*}
The function
\[
    f(r)=-r^2 +r (1+2(n-j)) +(k^2-k-2k(n-j))
\]
is a parabola with maximum at $r=\frac{1+2(n-j)}{2}=n-j+\frac{1}{2}$. In particular, it is increasing for $r<k \leq n-j < n-j+\frac{1}{2}$. So, for $r=k-1$, we have
\[
f(k-1) = 2 (k+j-n-1).
\]
and for $r<k-1$
\[
f(r)\leq f(k-2) = 2 (2k+2j-2n-3).
\]
Therefore, for $r=k-1$ we obtain (note that $j-i-k+r= j-i-1\geq 0$ in this case),
\begin{multline*}
(-1)^{k-r} q^{\binom{r+1}{2}+\binom{k}{2}-rk} \binom{k}{r}_q \binom{n-i-k+r}{j-i-k+r}_q \binom{n-i}{j-i}_q^{-1}
\\
= - q^{-(n-k-j+1)}(1+O(q^{-1})),
\end{multline*}
and for $r<k-1$ we obtain
\[
(-1)^{k-r} q^{\binom{r+1}{2}+\binom{k}{2}-rk} \binom{k}{r}_q \binom{n-i-k+r}{j-i-k+r}_q \binom{n-i}{j-i}_q^{-1}=  O(q^{-(2n-2k-2j+3)}).
\]
So
\begin{multline*}
c_{ijk}\binom{n-i}{j-i}_q^{-1}=  1 - q^{-(n-k-j+1)}(1+O(q^{-1}))+ O(q^{-(2n-2k-2j+3)})
\\
=1 - q^{-(n-k-j+1)}(1+O(q^{-1})).
\end{multline*}
Since $n-k-j+1\geq 1$, we obtain $c_{ijk}\binom{n-i}{j-i}_q^{-1}=1+O(q^{-1})$.

\end{proof}

For $0\leq i\leq n$, let
$
    E^i=\text{span}(\{ 1_U:\, U\in S(i)\}).
$
In \cite{kantor1972incidence}, Kantor proved the following theorem:

\begin{theorem}[Kantor \cite{kantor1972incidence}]\label{thm:kantor}
 Let $0\leq k \leq \left\lfloor\frac{n}{2}\right\rfloor$ and $k\leq i\leq n-k$. Then
\[
 \text{rank}(A_{ik})=|S(k)|=\binom{n}{k}_q.
\]
In particular, the linear map $A_{ik}: E^{k}\to E^{i}$
is injective.
\end{theorem}

Let $\tilde{E}^0=E^0$, and, for $1\leq k \leq \left\lfloor\frac{n}{2}\right\rfloor$ , let
$\tilde{E}^k$ be the orthogonal complement in $E^k$ of the subspace $A_{k, k-1} E^{k-1}$. Note that $\dim(\tilde{E}^0)=1$ and, since by Theorem \ref{thm:kantor} $A_{k,k-1}$ is injective,  $\dim(\tilde{E}^k)=\dim(E^k)-\dim(E^{k-1})= \binom{n}{k}_q-\binom{n}{k-1}_q$ for all $1\leq k\leq \lfrac{n}{2}$.

\begin{proposition}\label{prop:ei_decomposition}
Let $0\leq i\leq n$. Then,
\[
E^{i}= \bigoplus_{\substack0\leq k\leq \left\lfloor\frac{n}{2}\right\rfloor, \,k\leq i \leq n-k} A_{i k}\tilde{E}^{k}.
\]
\end{proposition}
\begin{proof}
Assume first $i\leq \lfrac{n}{2}$. We argue by induction on $i$. For $i=0$ the claim holds trivially.
Let $i>0$. Note that, by Lemma \ref{lemma:product1}, we have for $k\leq i-1$,
\[
    A_{i,i-1} A_{i-1,k} \tilde{E}^k =A_{i k}\tilde{E}^k.
\]
By the induction hypothesis, we obtain
\begin{align*}
    E^i &= A_{i,i-i} E^{i-1} \bigoplus \tilde{E}^{i}
    \\
    &= A_{i,i-1}\left(\bigoplus_{0\leq k\leq \lfrac{n}{2}, \, k\leq i-1} A_{i-1,k} \tilde{E}^k\right) \bigoplus  \tilde{E}^{i}
    \\
    &=
    \left(\bigoplus_{0\leq k\leq \lfrac{n}{2}, \, k\leq i-1} A_{i k} \tilde{E}^k \right) \bigoplus  \tilde{E}^{i}
\\
 &=
 \bigoplus_{0\leq k\leq \lfrac{n}{2}, \, k\leq i} A_{i k} \tilde{E}^k
 =
 \bigoplus_{0\leq k\leq \lfrac{n}{2}, \, k\leq i\leq n-k} A_{i k} \tilde{E}^k.
\end{align*}

Now, let $i> \lfrac{n}{2}$. By Lemma \ref{lemma:product1}, we have for $k\leq i\leq n-k$,
\[
    A_{i,n-i} A_{n-i,k}\tilde{E}^k = A_{i k}\tilde{E}^k.
\]
By Theorem \ref{thm:kantor}, we have $E^i=A_{i,n-i} E^{n-i}$. Therefore, since $n-i\leq \lfrac{n}{2}$, we obtain
\begin{align*}
    E^i &= A_{i,n-i} E^{n-i}
    =
     A_{i,n-i} \left(\bigoplus_{0\leq k\leq \lfrac{n}{2}, \, k\leq i\leq n-k} A_{n-i,k} \tilde{E}^k\right)
     \\
     &=
     \bigoplus_{0\leq k\leq \lfrac{n}{2}, \, k\leq i\leq n-k} A_{i k} \tilde{E}^k.
\end{align*}
\end{proof}

\begin{remark}
The action of the general linear group $\gl{n,q}$ on the $i$-dimensional subspaces of $\FF_q^n$ induces a $\gl{n,q}$-module structure on 
$E^i$. The decomposition of $E^i$ in Proposition \ref{prop:ei_decomposition} is in fact its decomposition into irreducible subrepresentations.
Although we don't explicitly use this fact here, this was our motivation for the definition of the subspaces $\tilde{E}^k$.
\end{remark}

The following properties of the subspaces $\tilde{E}^k$ will be needed later:

\begin{lemma}\label{lemma:orthogonal1} Let $0\leq k\leq \lfrac{n}{2}$ and let $v\in \tilde{E}^k$. Then, $A_{j k} v =0$ for all $0\leq j<k$.
\end{lemma}
\begin{proof}
By the definition of $\tilde{E}^k$, we have
\[
     v^t (A_{k,k-1} u)=0
\]
for all $u\in E^{k-1}$. Hence, $v^t A_{k,k-1}=0$. That is,
\[
    A_{k-1,k} v = (v^t A_{k,k-1})^t=0.
\]
Now, let $0\leq j\leq k-1$. By Lemma \ref{lemma:product1} we have
\[
    A_{k,k-1} A_{k-1,j}= \binom{k-j}{k-j-1}_q A_{k j}.
\]
Transposing the equation, we obtain
\[
    A_{j k} v= \binom{k-j}{k-j-1}_q^{-1} A_{j,k-1} (A_{k-1,k} v)=0.
\]
\end{proof}

\begin{lemma}\label{lemma:orthogonal2}
Let $0\leq k\leq \left\lfloor\frac{n}{2}\right\rfloor$. Let $0\leq i<k\leq j\leq n$, and let $v\in \tilde{E}^k$. Then,
$
    A_{i j} A_{j k} v=0.
$
\end{lemma}
\begin{proof}

By Lemma \ref{lemma:product2}, we have
\[
    A_{k j} A_{j i}=\sum_{m=0}^i c_{k j i m} A_{k m} A_{m i}.
\]
Transposing the equation, we obtain
\[
    A_{i j} A_{j k} v= \sum_{m=0}^i c_{k j i m} A_{i m} A_{m k} v.
\]
By Lemma \ref{lemma:orthogonal1}, we have $A_{m k}v=0$ for all $m\leq i<k$. Therefore, we obtain
\[
    A_{i j} A_{j k}v=0,
\]
as wanted.
\end{proof}

\section{Proof of Theorem \ref{thm:diagonal_representation}}\label{sec:laplacian_pap_proof}\label{sec:proof_diag_form}

For $0\leq k\leq \left\lfloor\frac{n}{2}\right\rfloor$, let $B_k$ be a basis for $\tilde{E}^{k}$. Then, by Proposition \ref{prop:ei_decomposition},
\[
    B= \bigcup_{k=0}^{\left\lfloor\frac{n}{2}\right\rfloor}
    \bigcup_{v\in B_k} \{A_{i k} v :\, \max\{1,k\}\leq i\leq \min\{n-1,n-k\}\}
\]
is a basis for $C^0(\flagcomplex{n}{q})$. We can now prove Theorem \ref{thm:diagonal_representation}.

\begin{proof}[Proof of Theorem \ref{thm:diagonal_representation}]
Let $L=\lap{0}{n}{q}$.
Let $0\leq k\leq \lfrac{n}{2}$, $v\in B_k$ and $\max\{1,k\}\leq j\leq \min\{n-1,n-k\}$. By Proposition \ref{lemma:matrix_form_L0}, we have
\[
    L A_{j k} v =-\sum_{i=1}^{j-1} \binom{n-i}{j-i}_q^{-1} A_{i j} A_{j k} v + (n-2) A_{j k} v - \sum_{i=j+1}^{n-1} \binom{i}{j}_q^{-1} A_{i j} A_{j k} v.
\]

Let $1\leq i\leq n-1$. If $i<k$, we have, by Lemma \ref{lemma:orthogonal2},
$
    A_{i j} A_{j k} v=0. 
$
If $i> n-k$, then  $n-i<k\leq j\leq n-k < i$. So, by Lemma \ref{lemma:product2}, we have (after transposing the equation),
\[
    A_{n-i,i} A_{i j} = \sum_{m=0}^{n-i} c_{j,i,n-i,m} A_{n-i,m} A_{m j}.
\]
Therefore, by Lemma \ref{lemma:orthogonal2}, we obtain
\[
A_{n-i,i} A_{i j} A_{j k} v =  \sum_{m=0}^{n-i} c_{j,i,n-i,m} A_{n-i, m} A_{m j} A_{j k} v = 0.
\]
Since $A_{n-i,i}$ is invertible (by Theorem \ref{thm:kantor}), we obtain
$
    A_{i j} A_{j k}v =0.
$
Now, assume $k\leq i\leq n-k$. If $i< j$, we have by Lemma \ref{lemma:product2} and Lemma \ref{lemma:orthogonal1}, 
\[
    A_{i j} A_{j k} v = \sum_{m=0}^k c_{i j k m} A_{i m} A_{m k}v= c_{i j k k} A_{i k} v= c_{i j k} A_{i k} v.
\]
If $i>j$, we have by Lemma \ref{lemma:product1},
\[
    A_{i j} A_{j k} v= \binom{i-k}{j-k}_q A_{i k} v.
\]
Therefore, we obtain
\begin{multline*}
L A_{j k} v = -\sum_{i=\max\{1,k\}}^{j-1} c_{i j k} \binom{n-i}{j-i}_q^{-1} A_{i k} v + (n-2) A_{j k} v \\- \sum_{i=j+1}^{\min\{n-1,n-k\}} \binom{i}{j}_q^{-1} \binom{i-k}{j-k}_q A_{i k} v.
\end{multline*}
Thus, the subspace spanned by the vectors 
\[
\{A_{ik} v:\, \max\{1,k\}\leq i\leq \min\{n-1,n-k\}\}
\]
is invariant under $L$, and the matrix representation of the restriction of $L$ on this subspace is exactly
 $\tilde{L}_k$ (for the case $k=0$, note that, by \eqref{eq:coeffs1}, we have $c_{i j 0}=\binom{n-i}{j-i}_q$). Therefore,
the representation of $L$ with respect to the basis $B$ is the diagonal block matrix
\[
    \begin{pmatrix}
    L_0 &   & \\
     & \ddots & \\
     & & L_{\lfrac{n}{2}}
\end{pmatrix}
\]
where
\[
    L_k = I_{\dim(\tilde{E}^k)} \otimes \tilde{L}_k   = \begin{pmatrix}
    \tilde{L}_k &   & \\
     & \ddots & \\
     & & \tilde{L}_k
\end{pmatrix}
\]
for all $0\leq k\leq \lfrac{n}{2}$.
\end{proof}

From Theorem \ref{thm:diagonal_representation} we already obtain some interesting immediate consequences. First, we show that the second part of Corollary \ref{cor:pap0case} holds:

\begin{corollary}\label{cor:pap_part1}
Let $n\geq 3$. Then, for every $\epsilon>0$ there is an integer $q_0(n,\epsilon)$ such that, for every prime power $q\geq q_0(n,\epsilon)$, every eigenvalue $\lambda\neq 0,n-1$ of $\lap{0}{n}{q}$ satisfies
    $
        |\lambda-(n-2)|<\epsilon.
    $
\end{corollary}
\begin{proof}
By Theorem \ref{thm:diagonal_representation}, each eigenvalue of $\lap{0}{n}{q}$ is an eigenvalue of the matrix $\tilde{L}_k$, for some $0\leq k\leq \lfrac{n}{2}$.
It is easy to check that the eigenvalues of $\tilde{L}_0$ are $0$ with multiplicity $1$, and $n-1$ with multiplicity $n-2$. Moreover, for all $1\leq k\leq \lfrac{n}{2}$, by Lemma \ref{lemma:limit_of_qbinoms} and Lemma \ref{lemma:coeffs_limit}, the matrix $\tilde{L}_k$ tends element-wise to the upper triangular matrix:
\[
    \lim_{q\to\infty} (\tilde{L}_k)_{i j}=
    \begin{cases}
    n-2 & \text{ if } i=j,\\
    -1 & \text{ if } i<j,\\
    0 & \text{ if } i>j.
    \end{cases}
\]
Therefore, as $q$ tends to infinity, all eigenvalues of $\tilde{L}_k$ tend to $n-2$. %This already confirms the second part of Corollary \ref{cor:pap0case}.
\end{proof}

Another immediate consequence of Theorem \ref{thm:diagonal_representation} is the following bound on the number of distinct eigenvalues of $\lap{0}{n}{q}$.

\begin{corollary}\label{cor:distinct_eigens}
For every prime power $q\geq 2$, the number of distinct eigenvalues of $\lap{0}{n}{q}$ is at most $\lfrac{n^2}{4}+2$.
\end{corollary}
\begin{proof}
    The number of distinct eigenvalues of $\tilde{L}_0$ is exactly $2$.
    For $1\leq k\leq \lfrac{n}{2}$, the number of distinct eigenvalues of $\tilde{L}_k$ is at most $n-2k+1$ (since $\tilde{L}_k$ is an $(n-2k+1)\times (n-2k+1)$ matrix). Therefore, by Theorem \ref{thm:diagonal_representation}, the number of distinct eigenvalues of $\lap{0}{n}{q}$ is at most
    \[
        2 + \sum_{k=1}^{\lfrac{n}{2}} (n-2k+1) = \lfrac{n^2}{4}+2.
    \]
\end{proof}

In order to prove our main result, Theorem \ref{thm:main_result}, we need a more detailed study of the spectra of the matrices $\tilde{L}_k$, which is presented in the next section.

\section{Spectra of the matrices $\tilde{L}_k$}
\label{sec:spectra}

For $m,r\in\mathbb{N}$, let $\mathcal{M}_{m,r}(x)$ be the set consisting of all matrices $M=M(x)\in \Rea^{m\times m}$ depending on a variable $x$ that satisfy
\[
    M_{ij}=\begin{cases}
    0 & \text{ if } i=j,\\
    -1+\bigoh{x^{-1}}& \text{ if } i<j,\\
    -x^{-r(i-j)} \left(1+\bigoh{x^{-1}}\right) & \text{ if } i>j.
    \end{cases}
\]

Note that, by Lemma \ref{lemma:limit_of_qbinoms} and Lemma \ref{lemma:coeffs_limit}, we have $\tilde{L}_k-(n-2)I\in \mathcal{M}_{n-2k+1,k}(q)$ for all $1\leq k\leq \lfrac{n}{2}$.

First, we estimate the coefficients of the characteristic polynomial of a matrix in $\mathcal{M}_{m,r}(x)$:
\begin{lemma}\label{lemma:matrix_char_poly}
Let $M\in\mathcal{M}_{m,r}(x)$, and let $p(t)=\det(t I-M)=t^m+a_{m-1}t^{m-1}+\cdots + a_0$ be the characteristic polynomial of $M$. Then, $a_{m-1}=0$, and for $0\leq i<m-1$, 
\[
    a_i= 
        (-1)^{\frac{m-i}{2}} \binom{(m+i)/2}{i} x^{-r\cdot \frac{m-i}{2}}        \left(1+\bigoh{x^{-1}}\right)
\]
if $m-i$ is even, and
\[
       a_i= (-1)^{\frac{m-i-3}{2}} (m-i-1)\binom{(m+i-1)/2}{i} x^{-r \cdot \frac{m-i+1}{2}} (1+\bigoh{x^{-1}})
\]
if $m-i$ is odd.
\end{lemma}

For the proof of Lemma \ref{lemma:matrix_char_poly}, we will need the following auxiliary results. 
Let $S_m$ be the group of permutations on $m$ elements. 
For a permutation $\pi\in S_m$, let
\[
    c(\pi)= \sum_{\substack{i\in [m]:\\ \pi(i)>i}} (\pi(i)-i).
\]
Let
\[
    c_{\ell,m}= \min\{ c(\pi):\, \pi\in S_m,\, \pi
    \text{ has $\ell$ fixed points}\}
\]
and
\[
    \mathcal{C}_{\ell,m}= \{ \pi\in S_m :\, \pi \text{ has $\ell$ fixed points and } c(\pi)=c_{\ell,m}\}.
\]

\begin{lemma}\label{lemma:permutations}
Let $m\geq 2$ be an integer, and let $\ell\leq m$. Then,
\[
    c_{\ell,m}= \ufrac{m-\ell}{2}.
\]
Moreover, if $m-\ell$ is even, $\mathcal{C}_{\ell,m}$ consists of all the permutations whose disjoint cycle decomposition is of the form
\[
     (i_1,i_1+1) (i_2,i_2+1)\cdots (i_{\frac{m-\ell}{2}}, i_{\frac{m-\ell}{2}}+1)
\]
where $i_1,\ldots,i_{\frac{m-\ell}{2}}\in[m]$, and, if $m-\ell$ is odd, then $\mathcal{C}_{\ell,m}$ consists of all the permutations whose disjoint cycle decomposition is of the form
\[
    (j,j+1,j+2)(i_1,i_1+1) (i_2,i_2+1)\cdots (i_{\frac{m-\ell-3}{2}}, i_{\frac{m-\ell-3}{2}}+1),
\]
or
\[
    (j+2,j+1,j)(i_1,i_1+1) (i_2,i_2+1)\cdots (i_{\frac{m-\ell-3}{2}}, i_{\frac{m-\ell-3}{2}}+1),
\]
where $j,i_1,\ldots,i_{\frac{m-\ell-3}{2}}\in[m]$.
\end{lemma}
\begin{proof}
Let $\pi\in S_m$ be a permutation with exactly $\ell$ fixed points. Then, since $\sum_{i=1}^m \pi(i)=\sum_{i=1}^m i$, we have
\[
0 = \sum_{i=1}^m (\pi(i)-i)= \sum_{\substack{i\in[m]:\\ \pi(i)>i}} (\pi(i)-i) +\sum_{\substack{i\in[m]:\\ \pi(i)<i}} (\pi(i)-i).
\]
So
\[
c(\pi)=\sum_{\substack{i\in[m]:\\ \pi(i)<i}} (i-\pi(i)).
\]
We obtain
\[
2\cdot c(\pi)= 
\sum_{\substack{i\in[m]:\\ \pi(i)>i}} (\pi(i)-i) 
+
\sum_{\substack{i\in[m]:\\ \pi(i)<i}} (i-\pi(i)) 
=
\sum_{\substack{i\in[m]:\\ \pi(i)\neq i}} |\pi(i)-i|.
\]
Since $|\pi(i)-i|\geq 1$ for all non-fixed points $i$, we obtain
\[
    2\cdot c(\pi)\geq m-\ell.
\]
Therefore, since $c(\pi)$, is an integer, we obtain
\begin{equation}\label{eq:cpi}
    c(\pi)\geq \ufrac{m-\ell}{2}.
\end{equation}
Moreover, when $m-\ell$ is even, equality in \eqref{eq:cpi} is obtained exactly when $|\pi(i)-i|=1$ for all non fixed points $i\in[m]$. Therefore $c_{\ell,m}=\frac{m-\ell}{2}$, and $\mathcal{C}_{\ell,m}$ consists of all the permutations whose disjoint cycle decomposition is of the form
\[
     (i_1,i_1+1) (i_2,i_2+1)\cdots (i_{\frac{m-\ell}{2}}, i_{\frac{m-\ell}{2}}+1)
\]
where $i_1,\ldots,i_{\frac{m-\ell}{2}}\in[m]$.

If $m-\ell$ is odd, then equality in \eqref{eq:cpi}
is obtained exactly when $|\pi(i)-i|=1$ for all non fixed points $i\in[m]$ except exactly one, which satisfies $|\pi(i)-i|=2$. Therefore $c_{\ell,m}=\frac{m-\ell+1}{2}=\ufrac{m-\ell}{2}$, and $\mathcal{C}_{\ell,m}$ consists of all the permutations whose disjoint cycle decomposition is of the form
\[
    (j,j+1,j+2)(i_1,i_1+1) (i_2,i_2+1)\cdots (i_{\frac{m-\ell-3}{2}}, i_{\frac{m-\ell-3}{2}}+1),
\]
or
\[
    (j+2,j+1,j)(i_1,i_1+1) (i_2,i_2+1)\cdots (i_{\frac{m-\ell-3}{2}}, i_{\frac{m-\ell-3}{2}}+1),
\]
where $j,i_1,\ldots,i_{\frac{m-\ell-3}{2}}\in[m]$.

\end{proof}

As a consequence, we obtain:
\begin{corollary}\label{cor:permutations}
Let $m\geq 2$ be an integer, and let $\ell\leq m$. Then, if $m-\ell$ is even,
\[
    |\mathcal{C}_{\ell,m}|=\binom{(m+\ell)/2}{\ell},
\]
and, if $m-\ell$ is odd, then
\[
    |\mathcal{C}_{\ell,m}|=\binom{(m+\ell-1)/2}{\ell}\cdot (m-\ell-1).
\]
\end{corollary}
\begin{proof}
    We divide into two cases: if $m-\ell$ is even, then, by Lemma \ref{lemma:permutations}, $\mathcal{C}_{\ell,m}$ is the set of all the permutations that can be written as a product of disjoint cycles consisting of $\ell$ cycles of length $1$ (the $\ell$ fixed points) and $\frac{m-\ell}{2}$ cycles of length $2$, each having the form $(i,i+1)$ for some $i\in[m]$. If we order these cycles in ascending order, we can encode each such permutation as a tiling of an $m\times 1$ board by $\ell$ tiles of length $1$ and $\frac{m-\ell}{2}$ tiles of length $2$ (where two numbers in $[m]$ belong to the same cycle if and only if their corresponding cells in the board are covered by the same tile). We can represent each tiling as a binary word with $\ell$ zeroes (the $1$-tiles) and $\frac{m-\ell}{2}$ ones (the $2$-tiles). Therefore, we obtain
\[
    |\mathcal{C}_{\ell,m}|
    = \binom{\ell+ \frac{m-\ell}{2}}{\ell}
    =\binom{(m+\ell)/2}{\ell}.
\]
Now, assume that $m-\ell$ is odd. Then, by Lemma \ref{lemma:permutations}, $\mathcal{C}_{\ell,m}$ is the set of all the permutations that can be written as a product of disjoint cycles consisting of $\ell$ cycles of length $1$, $\frac{m-\ell-3}{2}$ cycles of length $2$ of the form $(i,i+1)$ and one cycle of length $3$ of the form $(i,i+1,i+2)$ or $(i+2,i+1,i)$.
Ordering these cycles by ascending order, we can encode each permutation as a tiling of an $m\times 1$ board by $\ell$ tiles of length $1$, $\frac{m-\ell-3}{2}$ tiles of length $2$, and $1$ tile of length $3$. Since such tilings do not encode the orientation of the $3$-cycle, $|\mathcal{C}_{\ell,m}|$ is exactly twice the number of such tilings.

We can represent each tiling as a ternary word, with $\ell$ zeroes (the $1$-tiles), $\frac{m-\ell-3}{2}$ ones (the $2$-tiles) and $1$ two (the $3$-tile). Therefore,
\begin{align*}
|\mathcal{C}_{\ell,m}|= 2 \cdot \binom{ \ell + \frac{m-\ell-3}{2}+1}{\ell,\frac{m-\ell-3}{2},1}
&=2 \cdot\binom{(m+\ell-1)/2}{\ell}\cdot \frac{m-\ell-1}{2}
    \\ &=\binom{(m+\ell-1)/2}{\ell}\cdot (m-\ell-1).
\end{align*}
\end{proof}

\begin{proof}[Proof of Lemma \ref{lemma:matrix_char_poly}]
Let $B=tI-M$. 
We have
\[
p(t)=\det(B)= \sum_{\pi\in S_m} \text{sgn}(\pi) \prod_{i=1}^m B_{\pi(i),i}.
\]
Since the variable $t$ appears only in the diagonal entries of $B$, and $B_{ii}=t$ for all $1\leq i\leq m$, we can write $p(t)=\sum_{i=0}^m a_i t^i$, where
\[
a_i=\sum_{\substack{\pi\in S_m:\\ \pi \text{ has $i$ fixed points}}} \text{sgn}(\pi) \prod_{i\in[m]:\, \pi(i)\neq i} B_{\pi(i),i}.
\]
In particular, $a_m=1$ and $a_{m-1}=0$ (since there are no permutations in $S_m$ with exactly $m-1$ fixed points). By the definition of $\mathcal{M}_{m,r}$, we have
\[
    B_{\pi(i),i}=-M_{\pi(i),i}= x^{-r(\pi(i)-i)} (1+O(x^{-1}))
\]
if $\pi(i)>i$, and
\[
    B_{\pi(i),i}=-M_{\pi(i),i}= 1+O(x^{-1})
\]
if $\pi(i)<i$.
Therefore, for $0\leq i< m-1$, we obtain
\[
a_i=\sum_{\substack{\pi\in S_m:\\ \pi \text{ has $i$ fixed points}}} \text{sgn}(\pi)\cdot x^{-r\cdot c(\pi)} (1+O(x^{-1})).
\]
We divide into two cases. If $m-i$ is even, then, by Lemma \ref{lemma:permutations}, we have $c(\pi)\geq \frac{m-i}{2}$ for all $\pi\in S_m$ with $i$ fixed points. Moreover, all the permutations obtaining this bound satisfy $\text{sgn}(\pi)= (-1)^{\frac{m-i}{2}}$. By Corollary \ref{cor:permutations}, there are exactly $\binom{(m+i)/2}{i}$ such permutations. Therefore, we obtain
\[
a_i=(-1)^{\frac{m-i}{2}} \binom{(m+i)/2}{i} x^{-r \cdot \frac{m-i}{2}} (1+O(x^{-1})).
\]
If $m-i$ is odd, then, by Lemma \ref{lemma:permutations}, we have $c(\pi)\geq \frac{m-i+1}{2}$ for all $\pi\in S_m$ with $i$ fixed points. Moreover, all the permutations obtaining this bound satisfy $\text{sgn}(\pi)=(-1)^{\frac{m-i-3}{2}}$. By Corollary \ref{cor:permutations}, there are exactly $(m-i-1)\binom{(m+i-1)/2}{i}$ such permutations. Therefore, we obtain
\[
a_i=(-1)^{\frac{m-i-3}{2}} (m-i-1)\binom{(m+i-1)/2}{i} x^{-r \cdot \frac{m-i+1}{2}} (1+O(x^{-1})).
\]
\end{proof}

In order to compute the roots of the characteristic polynomial, we introduce the following change of variable:
\begin{lemma}\label{lemma:change_of_variables_1}
Let $m\geq 2$ and $r\geq 1$. Let $M\in\mathcal{M}_{m,r}(x)$, and let $p(t)$ be its characteristic polynomial. Let $\alpha=\min\{r/2,1\}$. Then,
\[
    p(s\cdot x^{-\frac{r}{2}}) = x^{-\frac{r m}{2}}\left( F_m(s) + h(s,x)
    \right),
\]
where
\[
    F_m(s)= \sum_{j=0}^{\lfrac{m}{2}} (-1)^j \binom{m-j}{j} s^{m-2j}
\]
and $h(s,x)$ is a polynomial of degree at most $m-2$ in $s$ whose coefficients are in $\bigoh{x^{-\alpha}}$.
\end{lemma}
\begin{proof}
    Let $0\leq i\leq m-2$ and let $t=s\cdot x^{-\frac{r}{2}}$. By Lemma \ref{lemma:matrix_char_poly}, if $m-i$ is even, we have
\begin{multline*}
    a_i t^i = (-1)^{\frac{m-i}{2}} \binom{(m+i)/2}{i} x^{-\frac{r m}{2}} (1+ O(x^{-1})) s^{i}
    \\
    =(-1)^{\frac{m-i}{2}} \binom{(m+i)/2}{i} x^{-\frac{r m}{2}}s^{i} + O(x^{\frac{-r m}{2}-1}) s^i,
\end{multline*}
and if $m-i$ is odd, then
\begin{multline*}
    a_i t^i = (-1)^{\frac{m-i-3}{2}} (m-i-1)\binom{(m+i-1)/2}{i} x^{-\frac{r (m+1)}{2}} (1+ O(x^{-1})) s^{i}
    \\
    = O(x^{-\frac{rm}{2}-\frac{r}{2}}) s^i. 
\end{multline*}
In addition, $p(t)$ has one extra term: $t^m=x^{-\frac{r m}{2}} s^m$.
Therefore, we can write
\[
    p(t)= p(s\cdot x^{-\frac{r}{2}})= x^{-\frac{r\cdot m}{2}} \left( F_m(s) +h(s,x)\right),
\]
where
\[
    F_m(s)= s^m + \sum_{\substack{0\leq i\leq m-2,\\ m-i \text{ is even}}} (-1)^{\frac{m-i}{2}} \binom{(m+i)/2}{i} s^i,
\]
and $h(s,x)$ is a polynomial of degree at most $m-2$ in $s$, whose coefficients are either in $O(x^{-1})$ or $O(x^{-\frac{r}{2}})$, and therefore are in $O(x^{-\alpha})$.

Finally, using the change of summation index $j=(m-i)/2$, we obtain
\[
   F_m(s)= \sum_{j=0}^{\lfrac{m}{2}} (-1)^j \binom{m-j}{j} s^{m-2j}.
\]
\end{proof}

The polynomials $F_m(s)$ are closely related to the sign-alternating Fibonacci polynomials studied by Donnelly et al. in \cite{donnellysign}. Let
\[
    G_m(s)=\sum_{j=0}^{\lfrac{m}{2}} (-1)^j \binom{m-j}{j} s^{\lfrac{m}{2}-j}
\]
be the $m$-th sign-alternating Fibonacci polynomial. The roots of $G_m(s)$ were determined in \cite{donnellysign}:

\begin{theorem}[Donnelly, Dunkum, Huber, Knupp {\cite[Corollary 3.2]{donnellysign}}]\label{thm:signalternatingfibo}
Let $m\geq 2$. Then, the set of roots of $G_m(s)$ is
\[
    \left\{ 4 \cos^2\left(\frac{j \pi}{m+1}\right) :\,  1\leq j\leq \lfrac{m}{2}\right\}.
\]
\end{theorem}

As a consequence, we obtain

\begin{corollary}\label{cor:signalternatingfibo}
Let $m\geq 2$, and let
\[
    \mathcal{I}_m= \left\{ \pm 2 \cos\left(\frac{j \pi}{m+1}\right) :\, 1\leq j\leq \lfrac{m}{2}\right\}.
\]  
Then, the set of roots of $F_m(s)$ is $\mathcal{I}_m$ if $m$ is even, and $\mathcal{I}_m\cup\{0\}$ if $m$ is odd.
\end{corollary}
\begin{proof}
Let $m$ be even. Then $F_m(s)=G_m(s^2)$. Therefore, by Theorem \ref{thm:signalternatingfibo}, the set of roots of $F_m(s)$ is $\mathcal{I}_m$.

Now, let $m$ be odd. Then $F_m(s)=s\cdot G_m(s^2)$. So, by Theorem \ref{thm:signalternatingfibo}, the set of roots of $F_m(s)$ is $\mathcal{I}_m$, plus the additional root $0$.
\end{proof}

In the case when $m$ is odd, we will also use the following change of variables:
\begin{lemma}\label{lemma:change_of_variables_2}
Let $m\geq 3$ be an odd integer, and let $r\geq 1$. Let $M\in\mathcal{M}_{m,r}(x)$, and let $p(t)$ be its characteristic polynomial. Then,
\[
    p(s\cdot x^{-r}) = (-1)^{\frac{m-1}{2}}\cdot  x^{-\frac{r (m +1)}{2}}\left( \frac{m+1}{2}\cdot s-(m-1) + g(s,x)
    \right),
\]
where $g(s,x)$ is a polynomial of degree $m$ in $s$ whose coefficients are in $\bigoh{x^{-1}}$. 
\end{lemma}
\begin{proof}
   Let $0\leq i\leq m-2$ and let $t=s\cdot x^{-r}$. By Lemma \ref{lemma:matrix_char_poly}, if $i$ is odd, then $m-i$ is even, so we have
\[
    a_i t^i = (-1)^{\frac{m-i}{2}} \binom{(m+i)/2}{i} x^{-\frac{r(m+i)}{2}} (1+\bigoh{x^{-1}}) s^i.
\]
For $i=1$, we get
\[
    a_1 t =  x^{-\frac{r(m+1)}{2}} \left( (-1)^{\frac{m-1}{2}}\cdot \frac{m+1}{2} +\bigoh{x^{-1}}\right)\cdot s.
\]
For $i\geq 3$, we obtain
\[
    a_i t^i = x^{-\frac{r(m+1)}{2}}\cdot  O(x^{-r}) \cdot s^i= x^{-\frac{r(m+1)}{2}}\cdot  O(x^{-1}) \cdot s^i.
\]
If $i$ is even, then $m-i$ is odd, and we obtain by Lemma \ref{lemma:matrix_char_poly},
\[
    a_i t^i = (-1)^{\frac{m-i-3}{2}} (m-i-1) \binom{(m+i-1)/2}{i} x^{-\frac{r(m+i+1)}{2}} (1+O(x^{-1})) s^i.
\]
For $i=0$, we get
\begin{align*}
    a_0 &= x^{-\frac{r(m+1)}{2}} \left( (-1)^{\frac{m-3}{2}} \cdot(m-1) + O(x^{-1})\right)
    \\
    &= -x^{-\frac{r(m+1)}{2}} \left( (-1)^{\frac{m-1}{2}} \cdot(m-1) + O(x^{-1})\right).
\end{align*}
For $i\geq 2$ we obtain
\[
a_i t^i = x^{-\frac{r(m+1)}{2}}\cdot  O(x^{-r}) \cdot s^i=x^{-\frac{r(m+1)}{2}}\cdot  O(x^{-1}) \cdot s^i.
\]
In addition, we have the term
\[
    t^m=  x^{-r m } s^m 
    =x^{-\frac{r(m+1)}{2}} \cdot x^{-\frac{r(m-1)}{2}}  s^m
    = x^{-\frac{r(m+1)}{2}} \cdot O(x^{-1}) s^m.
\]
Therefore, we can write 
\[
    p(s\cdot x^{-r}) = (-1)^{\frac{m-1}{2}} x^{-\frac{r(m+1)}{2}} \left( \frac{m+1}{2} s - (m-1) + g(s,x)\right),
\]
where $g(s,x)$ is a polynomial of degree $m$ in $s$ whose coefficients are in $O(x^{-1})$. 
\end{proof}

\begin{lemma}\label{lemma:approximation}
   Let $\beta>0$. Let 
    $H(z,x)$ be a polynomial in $z$ with coefficients depending on the real parameter $x$.
    Assume that we can write $H(z,x)=\tilde{H}(z)+E(z,x)$, where $\tilde{H}$ is a polynomial in $z$, and 
    \[
    E(z,x)=\sum_{j=0}^m b_{j}(x) z^j,
    \]
    where $b_j(x)= O(x^{-\beta})$ for all $0\leq j\leq m$.
    Let $\zeta\in\mathbb{C}$ be a simple root of $\tilde{H}$.
     Then, there exist $x_0, C >0$ such that, for all $x\geq x_0$, $H(z,x)$ contains a unique root in the disk
    \[
        \left\{ z\in\mathbb{C} :\, |z-\zeta|\leq C\cdot x^{-\beta}\right\}.
    \]
\end{lemma}
\begin{proof}

Since $b_j(x)=O(x^{-\beta})$ for all $0\leq j\leq m$, there exist $\tilde{x}_0>0$ and $K>0$ such that, for all $x\geq \tilde{x}_0$ and all $0\leq j\leq m$, 
    \[
        |b_j(x)|\leq K \cdot x^{-\beta}.
    \]
    Let $\zeta$ be a simple root of $\tilde{H}(z)$. Since $\zeta$ is simple, we have $\tilde{H}'(\zeta)\neq 0$, 
    and by Taylor's theorem, 
    \[
    \tilde{H}(z)= \tilde{H}'(\zeta)(z-\zeta)+f(z)(z-\zeta),
    \]
    where $f(z)$ is some function satisfying $\lim_{z\to\zeta} f(z)=0$.
    Let 
    \[
        C> \frac{2 K}{|\tilde{H}'(\zeta)|}\cdot\sum_{j=0}^m(|\zeta|+1)^j, 
    \]
    and assume $|z-\zeta|= C x^{-\beta}$. Then, we have
    \[
        |\tilde{H}(z)|\geq C x^{-\beta} \cdot |\tilde{H}'(\zeta)|- C x^{-\beta}\cdot |f(z)|.
    \]
    Since as $x\to \infty$ we have $z\to \zeta$,  there is some $\tilde{x}_1>0$ such that, for $x\geq \tilde{x}_1$, $|f(z)|\leq |\tilde{H}'(\zeta)|/2$, and therefore
    \[
        |\tilde{H}(z)|\geq C x^{-\beta} \cdot \frac{|\tilde{H}'(\zeta)|}{2}.
    \]
    On the other hand, for $x\geq \tilde{x}_0$, 
    \[
        |E(z,x)|\leq \sum_{j=0}^m |b_j(x)|\cdot|z|^j \leq \sum_{j=0}^m K x^{-\beta} |z|^j.
    \]
    For $x\geq C^{1/\beta}$ we have
    \[
        |z|\leq |z-\zeta|+|\zeta| = C x^{-\beta} +|\zeta|\leq |\zeta|+1.
    \]
    So, for $x\geq \max\{\tilde{x}_0,C^{1/\beta}\}$, 
    \[
        |E(z,x)|\leq K x^{-\beta} \sum_{j=0}^m (|\zeta|+1)^j.
    \]
   Thus, for $x\geq \max\{\tilde{x}_0,\tilde{x}_1, C^{1/\beta}\}$, we have
    $
        |\tilde{H}(z)|> |E(z,x)|
    $
    for all $z\in\mathbb{C}$ such that $|z-\zeta|=C x^{-\beta}$.
    Hence, by Rouch\'e's theorem (see e.g. \cite{sarason2007complex}), $H(z,x)$ and $\tilde{H}(z)$ have the same number of roots in the disk 
  \[
       D(x)= \left\{ z\in\mathbb{C} :\, |z-\zeta|\leq C\cdot x^{-\beta}\right\}.
    \]
    Since $\zeta$ is a simple root of $\tilde{H}(z)$, then there is some $\tilde{x}_2>0$ such that, for $x\geq \tilde{x}_2$, $\zeta$ is the unique root of $\tilde{H}(z)$ in $D(x)$. Let $x_0=\max\{\tilde{x}_0,\tilde{x}_1,\tilde{x}_2,C^{1/\beta}\}$.  Then, for $x\geq x_0$, there is a unique root of $H(z,x)$ in the disk $D(x)$.
\end{proof}

\begin{proposition}\label{prop:M_matrix_eigen}
Let $m\geq 2$ and $r\geq 1$.
Let $M\in\mathcal{M}_{m,r}(x)$, and let $\alpha=\min\{r/2,1\}$.
Then, there exist $x_0>0$ such that the following holds: 
for every $\zeta\in \mathcal{I}_m$ there is some $C_{\zeta}>0$ such that, for all $x\geq x_0$, there is exactly one eigenvalue of $M$ in the disk
\[
    \left\{ z\in \mathbb{C}:\, |z-\zeta\cdot x^{-\frac{r}{2}}|\leq C_{\zeta}\cdot x^{-(r/2+\alpha)}\right\}.
\]
Moreover, if $m$ is odd, there is some $C>0$ such that, for all $x\geq x_0$, there is exactly one eigenvalue of $M$ in the disk
\[
    \left\{z\in \mathbb{C}:\,
    \left| z-\frac{2(m-1)}{m+1}\cdot x^{-r}\right|\leq C \cdot x^{-(r+1)}
    \right\}.
\]
\end{proposition}
\begin{proof}
Let $p(t)$ be the characteristic polynomial of $M$. Let $\alpha=\min\{r/2,1\}$.
By Lemma \ref{lemma:change_of_variables_1}, we have
\begin{equation}\label{eq:p(sx)}
    p(s\cdot x^{-\frac{r}{2}}) = x^{-\frac{r m}{2}}\left( F_m(s) + h(s,x)
    \right),
\end{equation}
where
\[
    F_m(s)= \sum_{j=0}^{\lfrac{m}{2}} (-1)^j \binom{m-j}{j} s^{m-2j}
\]
and $h(s,x)$ is a polynomial of degree at most $m-2$ in $s$ whose coefficients are in $O(x^{-\alpha})$.
Let $\zeta\in \mathcal{I}_m$. By Corollary \ref{cor:signalternatingfibo}, $\zeta$ is a simple root of $F_m(s)$. Hence, by Lemma \ref{lemma:approximation}, there exist $x_{\zeta}>0$ and $C_{\zeta}>0$ such that, for $x\geq x_{\zeta}$, $F_m(s)+h(s,x)$ contains a unique root in the disk
\[
    \{ s\in \mathbb{C} :\, |s-\zeta|\leq C_{\zeta}\cdot x^{-\alpha}\}.
\]
Therefore, by \eqref{eq:p(sx)}, for $x\geq x_{\zeta}$,  $p(t)$ has exactly one root in the disk
\[
    D=\{ t\in \mathbb{C} :\, |t-\zeta\cdot x^{-\frac{r}{2}}|\leq C_{\zeta}\cdot x^{-\frac{r}{2}-\alpha}\}.
\]
That is, $M$ has exactly one eigenvalue in $D$. 

Now, assume that $m$ is odd. By Lemma \ref{lemma:change_of_variables_2}, we can write
\begin{equation}\label{eq:p(sx)_2}
  p(s\cdot x^{-r}) = (-1)^{\frac{m-1}{2}}\cdot  x^{-\frac{r(m +1)}{2}}\left( \frac{m+1}{2}\cdot s-(m-1) + g(s,x)
    \right),
\end{equation}
where $g(s,x)$ is a polynomial of degree $m$ in $s$ whose coefficients are in $O(x^{-1})$. Since $\frac{2(m-1)}{m+1}$ is a simple root of the polynomial $\frac{m+1}{2} s-(m-1)$, there exist by Lemma \ref{lemma:approximation} constants $\tilde{x}_0, C>0$ such that, for $x\geq \tilde{x}_0$, $
    \frac{m+1}{2} s-(m-1) + g(s,x)
$ 
contains a unique root in the disk
\[
    \left\{ s\in \mathbb{C} :\, \left|s-\frac{2(m-1)}{m+1}\right|\leq C\cdot x^{-1}\right\}.
\]
Thus, by \eqref{eq:p(sx)_2}, $p(t)$ contains exactly one root in the disk
\[
    D_2=\left\{ t\in \mathbb{C} :\, \left|t-\frac{2(m-1)}{m+1}\cdot x^{-r}\right|\leq C\cdot x^{-r-1}\right\}.
\]
That is, $M$ has exactly one eigenvalue in $D_2$.

Finally, the claim follows by letting $x_0$ be the maximum of all $x_{\zeta}$ (and $\tilde{x}_0$ in the odd $m$ case).

\end{proof}

We will need also the following simple lemma giving a sufficient condition for certain families of disks to be pairwise disjoint:

\begin{lemma}\label{lemma:disjoint_disks}
    Let $D_1(x),\ldots,D_m(x)\subset \mathbb{C}$ be a family of disks, depending on a real parameter $x$, defined by
    \[
    D_i(x)= \left\{ z\in\mathbb{C}:\, |z-(a+\zeta_i x^{-\alpha_i})|\leq C_i x^{-\beta_i}\right\},
    \]
    where $a,\zeta_i\in \mathbb{C}$, $C_i>0$ and $0<\alpha_i<\beta_i$ for all $1\leq i\leq m$.
    Assume that $\zeta_i x^{-\alpha_i}\neq \zeta_j x^{-\alpha_j}$ for all $1\leq i<j\leq m$.
    Then, there exists $x_0>0$ such that, for all $x\geq x_0$, $D_i(x)\cap D_j(x)=\emptyset$ for all $1\leq i<j\leq m$.
\end{lemma}
\begin{proof}
    Let $z_i=a+\zeta_i x^{-\alpha_i}$ and $R_i=C_i x^{-\beta_i}$ for all $1\leq i\leq m$. 
    Let $1\leq i<j\leq m$. Let $\alpha=\min\{\alpha_i,\alpha_j\}$ and $\beta=\min\{\beta_i,\beta_j\}$. Note that $\alpha<\beta$. Since $z_i\neq z_j$, there exists $c>0$ such that, for large enough $x$, 
    \[
        |z_i-z_j|\geq \left| |\zeta_i|x^{-\alpha_i}-|\zeta_j|x^{-\alpha_j}\right| > c x^{-\alpha}.
    \]
    On the other hand, 
    \[
    R_i+R_j\leq (C_i+C_j) x^{-\beta}.
    \]
    Since $\alpha<\beta$, we obtain for large enough $x$,
    \[
        |z_i-z_j|> R_i+R_j,
    \]
    and therefore $D_i(x)\cap D_j(x)=\emptyset$, as wanted.
\end{proof}

Finally, we can prove our main result, Theorem \ref{thm:main_result}.

\begin{proof}[Proof of Theorem \ref{thm:main_result}]
    By Theorem \ref{thm:diagonal_representation}, the spectrum of $\lap{0}{n}{q}$ is the union of the spectra of the matrices $\tilde{L}_k$, each repeated $\binom{n}{k}_q-\binom{n}{k-1}_q$ times, for $0\leq k\leq \lfrac{n}{2}$. As we already mentioned in the proof of Corollary \ref{cor:pap_part1}, the eigenvalues of $\tilde{L}_0$ are $0$ with multiplicity $1$, and $n-1$ with multiplicity $n-2$. Since $\binom{n}{0}_q-\binom{n}{-1}_q=1$, we obtain that $0$ is an eigenvalue of $\lap{0}{n}{q}$ with multiplicity $1$, and $n-1$ is an eigenvalue with multiplicity $n-2$.
    
    In addition, if $n$ is even, then by \eqref{eq:lap_block_formula}, $\tilde{L}_{n/2}$ is a $1\times 1$ matrix with a unique element $n-2$. Therefore, $n-2$ is an eigenvalue of $\lap{0}{n}{q}$ with multiplicity $\binom{n}{n/2}_q-\binom{n}{n/2-1}_q$.

    For $1\leq k\leq \lfrac{n-1}{2}$, note that, by Lemma \ref{lemma:limit_of_qbinoms} and Lemma \ref{lemma:coeffs_limit}, we have $\tilde{L}_k-(n-2)I\in \mathcal{M}_{n-2k+1,k}(q)$. Let $\alpha=\min\{k/2,1\}$.
    Then, by Proposition \ref{prop:M_matrix_eigen} (and noting that the eigenvalues of $\lap{0}{n}{q}$ are all real numbers), there exist $\tilde{q}_k>0$ such that for every $\zeta\in \mathcal{I}_{n-2k+1}=\mathcal{J}_k$ there is some $C_{\zeta}>0$ such that, for $q\geq \tilde{q}_k$, $\tilde{L}_k$ contains exactly one eigenvalue in the interval
    \[
     \{\lambda\in\Rea:\, |\lambda-(n-2+\zeta\cdot q^{-k/2})|\leq C_{\zeta}\cdot q^{-(k/2+\alpha)}\},
    \]
    and if $n$ is even, there is some $C_k>0$ such that $\tilde{L}_k$ contains exactly one eigenvalue in the interval
\[  
    \left\{\lambda\in \mathbb{R}:\, \left|\lambda-\left(n-2+\frac{2(n-2k)}{n-2k+2}\cdot q^{-k}\right)\right|\leq C_k\cdot q^{-(k+1)}\right\}.
\]
Let $\tilde{q}_0(n)=\max\{\tilde{q}_k:\, 1\leq k\leq \lfloor (n-1)/2 \rfloor\}$ and let $C(n)$ be the maximum of all $C_{\zeta}$ and $C_k$. Then, for $q\geq \tilde{q}_0(n)$, $\lap{0}{n}{q}$ contains at least one eigenvalue, with multiplicity $\binom{n}{k}_q-\binom{n}{k-1}_q$, in the interval
  \[
     D_{n,k,\zeta}(q)=\{\lambda\in\Rea:\, |\lambda-(n-2+\zeta\cdot q^{-k/2})|\leq C(n) \cdot q^{-(k/2+\alpha)}\},
    \]
for all $1\leq k\leq \lfrac{n-1}{2}$ and $\zeta\in\mathcal{J}_k$. In addition, if $n$ is even, $\lap{0}{n}{q}$ contains at least one eigenvalue, of multiplicity $\binom{n}{k}_q-\binom{n}{k-1}_q$, in the interval
\[  
    \tilde{D}_{n,k}(q)=\left\{\lambda\in \mathbb{R}:\, \left|\lambda-\left(n-2+\frac{2(n-2k)}{n-2k+2}\cdot q^{-k}\right)\right|\leq C(n)\cdot q^{-(k+1)}\right\},
\]
for all $1\leq k\leq \lfrac{n-1}{2}$. 

If $n$ is odd, let $\mathcal{D}_{k}=\{D_{n,k,\zeta}(q) :\, \zeta\in\mathcal{J}_k\}$. If $n$ is even, let $\mathcal{D}_k=\{\tilde{D}_{n,k}(q)\}\cup \{D_{n,k,\zeta}(q) :\, \zeta\in\mathcal{J}_k\}$. Let $\mathcal{D}=\bigcup_{k=1}^{\lfloor (n-1)/2\rfloor}\mathcal{D}_k$.
Note that the midpoints of the intervals in $\mathcal{D}$ are all distinct. Indeed, this is clear except in the case when $n$ is even, where we need to show that for all $1\leq k\leq (n/2-1)/2$ and $\zeta\in\mathcal{J}_k$, $\zeta\neq 2(n-2k)/(n-2k+2)$. But  $\zeta$ is (up to sign) twice the cosine of a rational multiple of $\pi$ lying strictly between $0$ and $\pi/2$;  hence, the only rational values that $\zeta$ can obtain are $\pm 1$ (see \cite[Cor. 3.12]{niven2005irrational}), and it is easy to check that the inequalities $1\neq 2(n-2k)/(n-2k+2)$ and $-1\neq 2(n-2k)/(n-2k+2)$ always hold for $1\leq k\leq (n/2-1)/2$. Therefore, by Lemma \ref{lemma:disjoint_disks}, there is some $q_0(n)>\tilde{q}_0(n)$ such that, for $q\geq q_0(n)$, the intervals in $\mathcal{D}$ are pairwise disjoint.

Finally, let $\mathcal{D}'=\mathcal{D}\cup\{\{0\},\{n-1\}\}$ if $n$ is odd, and $\mathcal{D}'=\mathcal{D}\cup\{\{0\},\{n-1\},\{n-2\}\}$ if $n$ is even. Note that $|\mathcal{J}_k|=n-2k+1$ if $n$ is odd, and $|\mathcal{J}_k|=n-2k$ if $n$ is even. So, the total number of eigenvalues of $\lap{0}{n}{q}$ (counting multiplicities) in the union of the intervals in $\mathcal{D}'$ is at least
\begin{align*}
1+(n-2) &+\sum_{k=1}^{(n-1)/2} |\mathcal{J}_k|\left(\binom{n}{k}_q-\binom{n}{k-1}_q\right)
\\
&=2\sum_{k=1}^{(n-1)/2} \binom{n}{k}_q = \sum_{k=1}^{n-1}\binom{n}{k}_q=|\flagcomplex{n}{q}(0)|
\end{align*}
for odd $n$, and at least
\begin{multline*}
1+(n-2)+ \left(\binom{n}{n/2}_q-\binom{n}{n/2-1}_q\right)+\sum_{k=1}^{n/2-1}(|\mathcal{J}_k|+1)\left(\binom{n}{k}_q-\binom{n}{k-1}_q\right)
\\
=\binom{n}{n/2}_q+2\sum_{k=1}^{n/2-1}\binom{n}{k}_q=
\sum_{k=1}^{n-1}\binom{n}{k}_q=|\flagcomplex{n}{q}(0)|
\end{multline*}
for even $n$.
Since $|\flagcomplex{n}{q}(0)|$ is the total number of eigenvalues of $\lap{0}{n}{q}$, we must in fact have equality, and therefore, for every $1\leq k\leq \lfrac{n-1}{2}$, every interval in $\mathcal{D}_k$ must contain exactly one eigenvalue of $\lap{0}{n}{q}$, with multiplicity $\binom{n}{k}_q-\binom{n}{k-1}_q$.

\end{proof}

\begin{proof}[Proof of Corollary \ref{cor:pap0case}]

The second part of the statement is proved in Corollary \ref{cor:pap_part1} (and also follows immediately from Theorem \ref{thm:main_result}).
    We are left to prove the first part. The argument is similar to the one we presented for 
    Corollary \ref{cor:distinct_eigens}.
    We divide into two cases. If $n$ is odd, then by Theorem \ref{thm:main_result} there is some $q_0(n)$ such that, for $q\geq q_0(n)$, the number of distinct eigenvalues of $\lap{0}{n}{q}$ is exactly 
\begin{align*}
2+\sum_{k=1}^{(n-1)/2} |\mathcal{J}_k| &=2+\sum_{k=1}^{(n-1)/2}2\lfrac{n-2k+1}{2}
\\&=2+\sum_{k=1}^{(n-1)/2}(n-2k+1)
=\lfrac{n^2}{4}+2.
\end{align*}
Similarly, if $n$ is even, there is some $q_0(n)$ such that, for $q\geq q_0(n)$, the number of distinct eigenvalues of $\lap{0}{n}{q}$ is exactly 
\begin{align*}
3+ \sum_{k=1}^{n/2-1}\left(|\mathcal{J}_k|+1\right)
&=3+ \sum_{k=1}^{n/2-1}\left(2\lfrac{n-2k+1}{2}+1\right) 
\\ &= 
3+ \sum_{k=1}^{n/2-1}(n-2k+1) =\lfrac{n^2}{4}+2.
\end{align*}
\end{proof}

\section{Concluding remarks}

It would be interesting to extend the first part of Corollary \ref{cor:pap0case} to all values of $q$ (for $q$ a prime power), or at least to find an explicit bound for the minimal value of $q$ for which the claim holds. Similarly, it would be interesting to find explicit bounds for the constant $q_0(n)$ in Theorem \ref{thm:main_result}. Moreover, based on some numerical evidence, we believe that the constant $C(n)$ appearing in Theorem \ref{thm:main_result} may be replaced by some absolute constant $C$ (perhaps $C=2$).

The complex $\flagcomplex{n}{q}$ is an example of a building (see e.g \cite{abramenko2008buildings}). Garland's Theorem \ref{thm:garland} holds in fact not only for the complexes $\flagcomplex{n}{q}$ but for a much larger family of buildings.
It seems plausible that Papikian's conjecture may extend to these more general cases (see \cite[Conjecture 5.7]{papikian2008eigenvalues} for such a statement in the $0$-dimensional case).
It may be interesting to try to extend our results to other  families of buildings.

The higher dimensional cases of Papikian's conjecture remain open. In ongoing work, by extending some of the ideas presented here to the high dimensional setting, and relying on the theory of unipotent representations of the general linear group $GL(n,q)$ (see e.g. \cite{james1984representations}), we developed an algorithm for computing the eigenvalues of $\lap{k}{n}{q}$ in the limit $q\to\infty$. As a consequence, we were able to verify the second part of Conjecture \ref{conj:papikian} for several small values of $k$ and $n$.

\section*{Acknowledgements}
Part of this research was done during the author's Ph.D. studies at the Technion, Israel Institute of Technology, under the supervision of Prof. Roy Meshulam, and part of it was done while the author was a postdoctoral researcher at the Einstein Institute of Mathematics at the Hebrew University of Jerusalem.

I thank Roy Meshulam and Eran Nevo for many helpful discussions.

\bibliographystyle{abbrv}
\bibliography{biblio}

\end{document}